\documentclass[12pt]{amsart}
\usepackage{amssymb}

\newtheorem{theorem}{Theorem}[section]
\newtheorem{lemma}[theorem]{Lemma}

\newtheorem{definition}[theorem]{Definition}

\newtheorem{corollary}[theorem]{Corollary}
\newtheorem{proposition}[theorem]{Proposition}

\newtheorem{lem-def}[theorem]{Lemma-Definition}

\newcommand{\hooklongrightarrow}{\lhook\joinrel\longrightarrow}
\renewenvironment{proof}{{\bfseries Proof.}}{\qed}
\newcommand{\N}{\mathbb N}
\newcommand{\Z}{\mathbb Z}
\newcommand{\Q}{\mathbb Q}

\def\op{\operatorname}

\def\aa{\mathcal{A}}

\def\al{\alpha}
\def\as#1{\renewcommand\arraystretch{#1}}

\def\bs{\vskip.5cm}
\def\be{\beta}

\def\cfa{\left(\rho_i\right)_{i\in A}}

\def\com{\op{com}}

\def\cl#1{[#1]_\mu}

\def\dg{\op{Deg}}

\def\diso{\lower.4ex\hbox{$\downarrow$}\raise.4ex\hbox{\mbox{\scriptsize
$\wr$}}}
\def\dm{\Delta}

\def\e{\medskip}

\def\ep#1{\exp(\Pi i#1)}
\def\ep{\epsilon}

\def\erel{e_{\op{rel}}}

\def\g{\Gamma}
\def\ga{\gamma}

\def\gchi{\g_{\mu,\deg(\chi)}}
\def\gen#1{\big\langle\, {#1} \,\big\rangle}

\def\gg{\mathcal{G}}

\def\ggm{\mathcal{G}_\mu}
\def\ggmp{\mathcal{G}_{\mu'}}
\def\ggn{\mathcal{G}_\nu}

\def\ginf{\g_\aa}

\def\gm{\g_\mu}

\def\gmp{\g_{\mu'}}
\def\gn{\g_\nu}

\def\gq{\g_\Q}
\def\gr{\operatorname{gr}}
\def\ggv{\mathcal{G}_v}

\def\imp{\ \Longrightarrow\ }
\def\im{\op{Im}}

\def\inm{\op{in}_\mu}
\def\inn{\op{in}_\nu}
\def\inmp{\op{in}_{\mu'}}
\def\inr{\op{in}_\rho}
\def\iso{\ \lower.3ex\hbox{\as{.08}$\begin{array}{c}\lra\\\mbox{\tiny $\sim\,$}\end{array}$}\ }

\def\ka{\kappa}

\def\km{k_\mu}
\def\kn{k_\nu}
\def\kp{\op{KP}}
\def\kpi{\kp_\infty}
\def\kpm{\op{KP}(\mu)}
\def\kpn{\op{KP}(\nu)}

\def\kpr{\op{KP}(\rho)}

\def\kx{K[x]}

\def\La{\Lambda}

\def\lg{l\raise.6ex\hbox to.2em{\hss.\hss}l}

\def\lra{\,\longrightarrow\,}

\def\mi{m_\infty}

\def\mlv{MacLane-Vaqui\'e\ }

\def\mmu{\mid_\mu}
\def\mn{\op{Min}}

\def\mx{\op{Max}}

\def\orb{\hbox to  .3em{$\backslash$}\backslash}
\def\ord{\op{ord}}
\def\p{\mathfrak{p}}

\def\phmn{\Phi_{\mu,\nu}}

\def\ppa{\mathcal{P}_{\alpha}}
\def\pset{\mathcal{P}}

\def\rhi{\rho_\aa}

\def\sii{\ \Longleftrightarrow\ }

\def\smu{\sim_\mu}

\def\sp{\op{Spec}}

\def\t{\theta}

\newcounter{cs}
\stepcounter{cs}
\newcommand{\casos}{\begin{itemize}}
\newcommand{\fcasos}{\end{itemize}\setcounter{cs}{1}}

\newfont{\tit}{cmr12 scaled \magstep3}

\title[MacLane-Vaqui\'e chains of valuations]{MacLane-Vaqui\'e chains of valuations on a polynomial ring}

\makeatletter
\@namedef{subjclassname@2010}{%
  \textup{2010} Mathematics Subject Classification}

\author{Enric Nart}
\address{Departament de Matem\`{a}tiques,
         Universitat Aut\`{o}noma de Barcelona,
         Edifici C, E-08193 Bellaterra, Barcelona, Catalonia}
\email{nart@mat.uab.cat}
\thanks{Partially supported by grant MTM2016-75980-P from the Spanish MEC}
\date{}

\begin{document}

\begin{abstract}
Let $(K,v)$ be a valued field.
We reinterpretate some results of MacLane and Vaqui\'e on extensions of $v$ to valuations on the polynomial ring $\kx$. We introduce certain  MacLane-Vaqui\'e chains constructed as a mixture of ordinary and limit augmentations of valuations. Every valuation $\nu$ on $\kx$ is a limit (in a certain sense) of a countable \mlv chain. This chain underlying  $\nu$ is essentially unique and contains discrete arithmetic data yielding an explicit description of the graded algebra of $\nu$ as an algebra over the graded algebra of $v$. 
\end{abstract}

\maketitle

\section*{Introduction}
Let $(K,v)$ be a valued field. 
In a pioneering work, S. MacLane studied the
extensions of the valuation $v$ to the polynomial ring $\kx$ in one 
indeterminate, in the case $v$ discrete of rank one \cite{mcla,mclb}. 
MacLane proved that all extensions of $v$ to $\kx$ can be obtained as a limit of chains of augmented valuations:
\begin{equation}\label{depthintro}
\mu_0\ \stackrel{\phi_1,\ga_1}\lra\  \mu_1\ \stackrel{\phi_2,\ga_2}\lra\ \cdots
\ \lra\ \mu_{n-1} 
\ \stackrel{\phi_{n},\ga_{n}}\lra\ \mu_{n}\ \lra\ \cdots
\end{equation}
involving the choice of certain \emph{key polynomials} $\phi_n \in K[x]$ and elements $\ga_n$ belonging to some extension of the value group of $v$.

For valuations of arbitrary rank, different approaches to this problem were developed by V. Alexandru, N. Popescu and A. Zaharescu \cite{APZ}, by F.-V. Kuhlmann \cite{Kuhl}, and by 
F.J. Herrera, W. Mahboub, M.A. Olalla and M. Spivakovsky \cite{hos,inmos}. 

This paper surveys the extension of MacLane's approach to the general case, developed by M. Vaqui\'e \cite{VaqJap,Vaq,Vaq2}. 
In this general context, \emph{limit augmentations} and the corresponding \emph{limit key polynomials} appear as a new feature.

Starting with a valuation $\mu_0$ which admits key polynomials of degree one, we introduce \emph{\mlv chains} as in (\ref{depthintro}), constructed as a mixture of ordinary and limit augmentations satisfying certain technical condition (Definition \ref{defMLV}). 
 
 The main result, Theorem \ref{main}, states that all valuations $\nu$ on $\kx$ extending $v$ fall in one, and only one, of the following cases:\e

(a) \ After a finite number $r$ of augmentation steps, we get $\mu_r=\nu$.\e

(b) \ After a finite number $r$ of augmentation steps, $\nu$ is the stable limit of
some set of ordinary augmentations of $\mu_r$, defined by key polynomials of constant degree.\e

(c) \ It is the stable limit of a countably infinite chain of mixed augmentations as in (\ref{depthintro}), defined by key polynomials with unbounded degree.\e

We say that $\nu$ has \emph{finite depth} $r$, \emph{quasi-finite} depth $r$, or \emph{infinite depth}, respectively.


In section \ref{secGmu}, we use this approach to obtain an explicit description of the structure of the graded algebra of $\nu$ as an algebra over the graded algebra of $v$, in terms of discrete arithmetic data supported by the underlying MacLane-Vaqui\'e chain of $\nu$, which  is essentially unique.



\section{Key polynomials over valued fields}\label{secKP}
Let $(K,v)$ be a non-trivially valued field. Let $k$ be the residue class field, $\g=v(K^*)$ the value group and $\gq=\g\otimes\Q$ the divisible hull of $\g$.

Consider an extension of $v$ to the polynomial ring $\kx$ in one indeterminate. That is, for some embedding $\g\hookrightarrow\Lambda$ into another ordered abelian group, we consider a mapping whose restriction to $K$ is $v$,
$$
\mu\colon \kx\lra \Lambda\infty
$$
and satisfies the following two conditions for all $f,g\in\kx$:
$$\mu(fg)=\mu(f)+\mu(g),\qquad\mu(f+g)\ge\mn\{\mu(f),\mu(g)\}.$$


The \emph{support} of $\mu$ is the prime ideal \ $\p=\p_\mu=\mu^{-1}(\infty)\in\sp(\kx)$. 

The value group of $\mu$ is the subgroup $\gm\subset \Lambda$ generated by $\mu\left(\kx\setminus\p\right)$.



The valuation $\mu$ induces in a natural way a valuation $\bar{\mu}$ on the residue field $\ka(\p)$, the field of fractions of $\kx/\p$. Let $\km$ be the residue class field of $\bar{\mu}$.

Note that $\ka(0)=K(x)$, while $\ka(\p)$ is a simple finite extension of $K$ if $\p\ne0$. 

The extension $\mu/v$ is \emph{commensurable} if $\g_\mu/\g$ is a torsion group. In this case, there is a canonical embedding $\g_\mu\hookrightarrow \gq$.

All valuations with non-trivial support are commensurable over $v$.

\subsection{Graded algebra of a valuation}\label{subsecGr}
For any $\alpha\in\g_\mu$, consider the abelian groups:
$$
\ppa=\{g\in \kx\mid \mu(g)\ge \alpha\}\supset
\ppa^+=\{g\in \kx\mid \mu(g)> \alpha\}.
$$    
The \emph{graded algebra of $\mu$} is the integral domain:
$$
\ggm:=\gr_{\mu}(\kx)=\bigoplus\nolimits_{\alpha\in\g_\mu}\ppa/\ppa^+.
$$

There is a natural \emph{initial term} mapping $\inm\colon \kx\to \ggm$, given by $\inm\p=0$ and 
$$
\inm g= g+\pset_{\mu(g)}^+\in\pset_{\mu(g)}/\pset_{\mu(g)}^+, \qquad\mbox{if }\ g\in \kx\setminus\p.
$$

There is a natural embedding of graded algebras \
$\gg_v:=\op{gr}_v(K)\hookrightarrow \ggm$.

If $\mu$ has non-trivial support $\p\ne0$, there is a natural isomorphism of graded algebras
\begin{equation}\label{isomL}
\ggm\simeq \op{gr}_{\bar{\mu}}(\ka(\p)).
\end{equation}

The next definitions translate properties of the action of  $\mu$ on $\kx$ into algebraic relationships in the graded algebra $\ggm$.

\begin{definition}Let $g,\,h\in \kx$.

We say that $g,h$ are \emph{$\mu$-equivalent}, and we write $g\smu h$, if $\inm g=\inm h$. 

We say that $g$ is \emph{$\mu$-divisible} by $h$, and we write $h\mmu g$, if $\inm h\mid \inm g$ in $\ggm$.

We say that $g$ is $\mu$-irreducible if $(\inm g)\ggm$ is a non-zero prime ideal. 

We say that $g$ is $\mu$-minimal if $g\nmid_\mu f$ for all non-zero $f\in \kx$ with $\deg(f)<\deg(g)$.
\end{definition}

The property of $\mu$-minimality admits a relevant characterization.

\begin{lemma}\cite[Prop. 2.3]{KP}\label{minimal0}
\ Let $\phi\in \kx$ be a non-constant polynomial. Let 
$$
f=\sum\nolimits_{0\le s}a_s\phi^s, \qquad  a_s\in \kx,\quad \deg(a_s)<\deg(\phi) 
$$
be the canonical $\phi$-expansion of $f\in \kx$.
Then, $\phi$ is $\mu$-minimal if and only if
$$
\mu(f)=\min\{\mu(a_s\phi^s)\mid 0\le s\},\quad \forall\,f\in \kx.
$$
\end{lemma}

\subsection{Key polynomials}\label{subsecKP}
A (MacLane-Vaqui\'e) \emph{key polynomial} for $\mu$ is a monic polynomial in $\kx$ which is simultaneously  $\mu$-minimal and $\mu$-irreducible. 

The set of key polynomials for $\mu$ is denoted $\kpm$. 

All $\phi\in\kpm$ are irreducible in $\kx$.\e

If $\mu$ has non-trivial support, the isomorphism of (\ref{isomL}) shows that every non-zero homogeneous element of $\ggm$ is a unit. 
Thus, $\kpm=\emptyset$, because no polynomial in $\kx$ can be $\mu$-irreducible.

If $\kpm\ne\emptyset$, the \emph{degree}  $\deg(\mu)$ is the minimal degree of a key polynomial for $\mu$.



\begin{lemma}\cite[Prop. 3.5]{KP}\label{units}
If $\kpm\ne\emptyset$, a non-zero  homogeneous element $\inm f$ is a unit in $\ggm$ if and only if $f\smu a$, for some $a\in\kx$ with $\deg(a)<\deg(\mu)$. In this case, $\inm f$ is algebraic over $\gg_v$ and $\mu(f)$ belongs to $\gq$. 
\end{lemma}

\begin{theorem}\cite[Thm. 3.9]{KP}\label{univbound}
Let $\phi\in\kpm$. For any monic $f\in \kx\setminus K$, we have
$$
\mu(f)/\deg(f)\le \mu(\phi)/\deg(\phi),
$$
and equality holds if and only if $f$ is $\mu$-minimal.
\end{theorem}

For any positive integer $m$, consider the subset 
$$ \g_{\mu,m}=\left\{\mu(a)\mid 0\le \deg(a)<m\right\}\subset\gm.$$
For all $\chi\in\kpm$, the subset $\gchi$ is a subgroup of $\gm$ and $\gen{\gchi,\mu(\chi)}=\gm$.

\begin{definition}\label{defErel}
If $\kpm\ne\emptyset$, the \emph{relative ramification index} of $\mu$ is defined as:
$$
e:=\erel(\mu):=\left(\gm\colon \g_{\mu,\deg(\mu)}\right).
$$
\end{definition}

By Lemma \ref{units}, $\g_{\mu,\deg(\mu)}\subset\gq$. Hence, if $\mu/v$ is incommensurable, we have $e=\infty$.

If $\mu/v$ is commensurable, $e$ is the least positive integer such that $e\mu(\phi)\in\g_{\mu,\deg(\mu)}$, where $\phi$ is any key polynomial for $\mu$ of minimal degree.

\begin{definition}\label{defproper}
A key polynomial $\chi$ for $\mu$ is said to be \emph{proper} if there exists some $\phi\in\kpm$ of minimal degree such that $\chi\not\smu \phi$. 
\end{definition}

For $\chi\in\kpm$, denote by $\cl{\chi}\subset \kpm$ the subset of all key polynomials which are $\mu$-equivalent to $\chi$.

Since properness depends only on the class $[\chi]_\mu$, we may speak of \emph{proper} and \emph{improper} classes of $\mu$-equivalence of key polynomials for $\mu$. 

If $e=1$, all classes are proper. If $e>1$,  all key polynomials for $\mu$ of minimal degree are $\mu$-equivalent and form the unique improper class in $\kpm/\!\!\smu$ \cite[Cor. 6.5]{KP}.

\begin{lemma}\cite[Cor. 6.4]{KP}\label{groupchain0}
If  $\chi$ is a proper key polynomial  for $\mu$, then $\gchi=\gm$.
\end{lemma}

Consider the subring  of homogeneous elements of degree zero in the graded algebra
$$\Delta=\Delta_\mu=\pset_0/\pset_0^+\subset\ggm.$$ 
There are canonical injective ring homomorphisms $\ k\hookrightarrow\Delta\hookrightarrow \km$. 
 We denote  the algebraic closure of $k$ in $\Delta$ by  $$\kappa=\kappa(\mu)\subset\Delta.$$ This subfield satisfies $\kappa^*=\Delta^*$, where $\Delta^*$ is the multiplicative group of all units in $\Delta$.

The following results compute the ring $\Delta$ in terms of the presence (or absence) of key polynomials and the commensurability of $\mu$.

\begin{theorem}\cite[Thm. 4.4]{KP}\label{empty}
The set $\kpm$ is empty if and only if all homogeneus elements in $\ggm$ are units. Equivalently, $\mu/v$ is commensurable and $\ka=\dm=\km$ is an algebraic extension of $k$.
\end{theorem}

These valuations with $\kpm=\emptyset$ are said to be \emph{valuation-algebraic} \cite{Kuhl}. 
The valuations  with $\kpm\ne\emptyset$ are \emph{valuation-transcendental}. They split into two families.

\begin{theorem}\cite[Thm. 4.2]{KP}\label{incomm}
Suppose $\mu/v$ incommensurable. Let $\phi\in\kx$ be a monic polynomial of minimal degree satisfying $\mu(\phi)\not\in\gq$. Then,
$\phi$ is a key polynomial for $\mu$, and $\kpm=\cl{\phi}$.
In particular, all key polynomials for $\mu$ are improper.

In this case, $\ka=\dm=\km$ is a finite extension of $k$.
\end{theorem}

\begin{theorem}\cite[Thms. 4.5,4.6]{KP}\label{comm}
Suppose $\mu/v$ commensurable and $\kpm\ne\emptyset$.  
Let $\phi$ be a key polynomial for $\mu$ of minimal degree $m$. Let $e=\erel(\mu)$. 

Let $u=\inm a\in\ggm^*$, for some $a\in\kx$ such that $\deg(a)<m$ and $\mu(a)=e\mu(\phi)$. Then, 
$\xi=(\inm \phi)^eu^{-1}\in\Delta$
is transcendental over $k$ and satisfies $\Delta=\kappa[\xi]$.

Moreover, the canonical embedding $\Delta\hookrightarrow \km$ induces an isomorphism $\ka(\xi)\simeq \km$. 
\end{theorem}

These comensurable extensions $\mu/v$ admitting (MacLane-Vaqui\'e) key polynomials are called \emph{residually transcendental} valuations on $\kx$.

The pair $\phi,\, u$ determines a \emph{residual polynomial operator} 
$$
R=R_{\mu,\phi,u}\colon\;\kx\lra \kappa[y],
$$
which facilitates a complete description of the set $\kpm$, in terms of any fixed key polynomial of minimal degree.

\begin{theorem}\cite[Prop. 6.3]{KP}\label{charKP}
Suppose that $\mu$ is residually transcendental. Let $\phi\in\kpm$ of minimal degree $m$. 
A monic $\chi\in\kx$ is a key polynomial for $\mu$ if and  only if either
\begin{itemize}
\item $\deg(\chi)=m$ \,and\; $\chi\smu\phi$, or\vskip0.1cm
\item $\deg(\chi)=me\deg(R(\chi))$ \,and\; $R(\chi)$ is irreducible in $\kappa[y]$.
\end{itemize}
\end{theorem}

\begin{corollary}\cite[Prop. 6.6]{KP}\label{mid=sim}
For any two key polynomials $\chi,\,\chi'\in\kpm$, we have
$$
\chi\mmu \chi'\sii \chi\smu\chi'\sii R(\chi)=R(\chi').
$$
In this case, $\deg(\chi)=\deg(\chi')$.
\end{corollary}

\section{Chains of augmentations}\label{secAugm}

Let $\mu$ be a valuation on $\kx$ extending the valuation $v$ on $K$. 

Let  $\gm\hookrightarrow \Lambda$ be an embedding of ordered groups, and let $\nu$ be a $\Lambda$-valued valuation whose restriction to $K$ is $v$. 
We say that $\mu\le\nu$ \,if
$$
\mu(f)\le \nu(f),\qquad \forall\,f\in\kx.
$$

In this case, there is a canonical homomorphism of graded $\gg_v$-algebras:
$$\ggm\lra\ggn,\qquad \inm f\longmapsto
\begin{cases}\inn f,& \mbox{ if }\mu(f)=\nu(f),\\ 0,& \mbox{  if }\mu(f)<\nu(f).\end{cases}
$$

A valuation $\mu$ is said to be \emph{maximal} if it admits no strict upper bounds. In other words, $\mu\le\nu$ implies $\mu=\nu$.

There are two perspectives concerning the comparison of $\mu$ with other valuations taking values in $\La$. We may ``look forward" and construct augmentations of $\mu$, or we may ``look backward" and try to describe the set
$$
 (-\infty,\mu)_{\Lambda}=\left\{\rho\colon \kx\to\Lambda\infty\mid \rho\mbox{ valuation, }\rho_{\mid K}=v,\ \rho<\mu\right\},
$$
usually with the purpose of finding ``constructable" approximations to $\mu$ from below.

In the discrete-rank one case, MacLane designed concrete procedures in both directions \cite{mcla}, which were generalised by Vaqui\'e to the general case \cite[Sec. 1]{Vaq}.

These procedures are summarized in Propositions \ref{extension} and \ref{gap} below. 


\subsection{Ordinary augmentations}\label{subsecOrdinary}
Suppose $\kpm\ne\emptyset$.
Choose $\phi\in \kpm$ and $\ga\in\Lambda\infty$ such that $\mu(\phi)<\ga$. 


The \emph{augmented valuation} of $\mu$ with respect to this pair of data is the mapping 
$$\mu':\kx\lra \Lambda \infty,\qquad f=\sum\nolimits_{0\le s}a_s\phi^s\longmapsto \mu'(f)=\min\{\mu(a_s)+s\ga\mid 0\le s\},
$$
defined in terms of $\phi$-expansions.
We use the notation $\mu'=[\mu;\phi,\ga]$. Note that $\mu'(\phi)=\ga$.

\begin{proposition}{\cite[sec. 1.1]{Vaq}, \cite[sec. 7]{KP}}\label{extension} 
\mbox{\null}
\begin{enumerate}
\item The mapping $\mu'=[\mu;\phi,\ga]$ is a valuation on $\kx$. 

If $\ga<\infty$, it has trivial support. If $\ga=\infty$, the support of $\mu'$ is $\phi\kx$.  
\item It satisfies $\mu<\mu'$. Moreover, for a non-zero $f\in\kx$, we have $$\mu(f)=\mu'(f)\sii\phi\nmid_{\mu}f.$$ In this case,
$\op{in}_{\mu'}f$ is a unit in $\gg_{\mu'}$.
\item If $\ga<\infty$, then $\phi$ is a key polynomial for $\mu'$ of minimal degree. In particular, $\deg(\mu')=\deg(\phi)$.
\item If $\phi$ is a proper key polynomial for $\mu$, then
$$ \gmp=\gen{\gm,\ga},\ \mbox{ if }\ \ga<\infty;\qquad \gmp=\gm,\ \mbox{ if }\ \ga=\infty.
$$ 
\end{enumerate}
\end{proposition}

\begin{proposition} \cite[Thm. 1.15]{Vaq}\label{gap}
If $\rho<\mu$, let $\Phi_{\rho,\mu}$ be the set of monic polynomials $\phi\in\kx$ of minimal degree satisfying $\rho(\phi)<\mu(\phi)$. 

Then, $\Phi_{\rho,\mu}\subset\kpr$ and for any $\phi\in \Phi_{\rho,\mu}$ we have 
$$\rho<[\rho;\phi,\mu(\phi)]\le\mu.$$ For any non-zero $f\in\kx$, the equality $\rho(f)=\mu(f)$ holds if and only if $\phi\nmid_\rho f$.
\end{proposition}

The following observation is an immediate consequence of  Propositions \ref{extension} and \ref{gap}.

\begin{theorem}\label{maximal}
A valuation $\mu$ is maximal if and only if $\,\kpm=\emptyset$.
\end{theorem}

Theorem \ref{interval} is another relevant consequence of Proposition \ref{gap}. The proof given in \cite[Thm. 3.9]{defless} for the group $\La=\gq$ is valid for any ordered group $\La$.

\begin{theorem}\label{interval}
For any valuation $\mu\colon\kx\to\La\infty$, the set $(-\infty,\mu)_{\La}$ is totally ordered.
\end{theorem}

Let us derive some more practical consequences of Propositions \ref{extension} and \ref{gap}.

\begin{corollary}\label{Phiclass}
Take $\rho<\mu$ as above, and let $\nu$ be another valuation. Then,
\begin{enumerate}
\item $\Phi_{\rho,\mu}=[\phi]_\rho$, for all $\phi\in \Phi_{\rho,\mu}$.
\item $\rho<\mu<\nu\imp \Phi_{\rho,\mu}=\Phi_{\rho,\nu}$. In particular,
$$
\rho(f)=\nu(f)\sii \rho(f)=\mu(f),\qquad  \forall\,f\in\kx.
$$
\end{enumerate}
\end{corollary}

\begin{proof}
For all $\chi\in\kpr$, $\phi\in\Phi_{\rho,\mu}$, Proposition \ref{gap} and  Corollary \ref{mid=sim}  show that
$$
\chi\in\Phi_{\rho,\mu}\ \sii\ \phi\mid_{\rho}\chi\ \sii\ \phi\sim_\rho\chi.
$$ 

On the other hand, if $\rho<\mu<\nu$, let $\Phi_{\rho,\mu}=[\phi]_\rho$ $\Phi_{\rho,\nu}=[\phi']_\rho$. Since
$$
\rho(\phi)<\mu(\phi)\le\nu(\phi),
$$
Proposition \ref{gap} shows that $\phi'\mid_\rho \phi$. Thus, $\phi'\sim_\rho \phi$, by Corollary \ref{mid=sim}. 
\end{proof}

\begin{corollary}\label{KerImGr}
Suppose that $\mu<\nu$ and let $\phmn=\cl{\phi}$. Then,
\begin{enumerate}
\item The kernel of the homomorphism $\ggm\to\ggn$ is the prime ideal $(\inm\phi)\ggm$.
\item All non-zero homogeneous elements in the image of \ $\ggm\to\ggn$ are units.
\item  If $\op{KP}(\nu)\ne\emptyset$, then $\deg(\mu)\le \deg(\nu)$.
\end{enumerate}
\end{corollary}

\begin{proof} Let $\rho=[\mu;\phi,\nu(\phi)]$, and let $\inm f\in \ggm$ be a non-zero  homogeneous element.

This element is mapped to zero in $\ggn$ if and only if $\mu(f)<\nu(f)$, and this is equivalent to $\phi\mmu f$ by Proposition \ref{gap}. This proves (1).

By Proposition \ref{gap}, $\mu<\rho\le\nu$, and  $\ggm\to\ggn$ is the composition of the canonical homomorphisms $\ggm\to\gg_{\rho}\to\ggn$.

If $\inm f$ is not mapped to zero in $\ggn$, then $\mu(f)=\rho(f)=\nu(f)$. Hence, the image of $\inm f$ under the composition $\ggm\to\gg_{\rho}\to\ggn$ is
$\,\inm f\mapsto \inr f\mapsto \inn f$.

By Proposition \ref{extension}, $\inr f$ is a unit in $\gg_\rho$. Thus,  $\inn f$  is a unit. This proves (2).

Since $\phi$ is a key polynomial for $\mu$, we have $\deg(\mu)\le \deg(\phi)$. 

On the other hand, all $a\in\kx$ with $\deg(a)<\deg(\phi)$ satisfy $\mu(a)=\nu(a)$. Thus, $a$ cannot be a key polynomial for $\nu$. In fact, $\inn a$ is a unit in $\ggn$ by item (2), because it is the image of $\inm a$.
Therefore, $\deg(\phi)\le \deg(\nu)$ if $\kpn\ne\emptyset$. This proves (3).
\end{proof}


\begin{lemma}\label{intervalOrd}
Let $\mu'=[\mu;\,\phi,\ga]$ be an augmentation of $\mu$. For any $\delta\in \La$, $\delta>\mu(\phi)$, consider the augmented valuation $\mu_\delta=[\mu;\phi,\delta]$. Then,
$$(\mu,\mu')_{\La}:=\left\{\rho\colon\kx\to\La\infty\mid \rho \mbox{ valuation},\ \mu<\rho<\mu'\right\}=\{\mu_\delta\mid \mu(\phi)<\delta<\ga\}.
$$ 
\end{lemma}

\begin{proof}
For all $a\in\kx$ of degree less than $\deg(\phi)$ we have $\mu(a)=\rho(a)=\mu'(a)$ for all $\rho\in(\mu,\mu')_\La$, by the definition of the augmented valuation.

If $\mu(\phi)<\delta<\ga=\mu'(\phi)$, then $\mu<\mu_\delta<\mu'$, by their action on $\phi$-expansions. 

Conversely, take any $\rho\in (\mu,\mu')_{\La}$. By Corollary \ref{Phiclass} (2), $\mu(\phi)<\rho(\phi)$.  Also, we must have $\rho(\phi)<\ga$ because otherwise we would have $\rho\ge\mu'$, by their action on $\phi$-expansions. 

Since $\phi$ is a monic polynomial of minimal degree satisfying  $\mu(\phi)<\rho(\phi)$, Proposition \ref{gap} shows that $\phi$ is a key polynomial for $\rho$. Hence, $\rho=\mu_\delta$, for $\delta=\rho(\phi)$, because both valuations coincide on $\phi$-expansions.
\end{proof}

\subsection*{Unicity of ordinary augmentations}\label{subsecUniqOrd}

Let us analyze when different building data $\phi$, $\ga$ yield the same ordinary augmentation of $\mu$.

\begin{lemma}\label{unicityOrd}
Let $\phi,\phi_*\in \kpm$. Let $\gm\hookrightarrow\Lambda$ be an embedding of ordered groups, and choose $\ga,\ga_*\in\Lambda\infty$ such that $\mu(\phi)<\ga$ and  $\mu(\phi_*)<\ga_*$. 

Then, $\mu'=[\mu;\phi,\ga]$ coincides with $\mu_*'=[\mu;\phi_*,\ga_*]$ if and only if 
\begin{equation}\label{equalmu}
\deg(\phi)=\deg(\phi_*)\quad \mbox{ and }\quad\mu(\phi_*-\phi)\ge \ga=\ga_*.
\end{equation}
\end{lemma}

\begin{proof}
Suppose that the conditions of (\ref{equalmu}) hold. Let $\phi_*=\phi+a$ with $a\in\kx$ of degree less than $\deg(\phi)$.
By assumption, $\mu'_*(a)=\mu'(a)=\mu(a)\ge\ga=\mu'(\phi)$.

Hence, $\mu'(\phi_*)\ge\ga=\ga_*=\mu'_*(\phi_*)$. By comparison of their action on $\phi_*$-expansions we deduce that $\mu'\ge\mu'_*$. By the symmetry of (\ref{equalmu}),  $\mu'=\mu'_*$.\e

Conversely, suppose $\mu'=\mu'_*$. By Proposition \ref{extension}, $\phi$, $\phi_*$ are key polynomials for $\mu'$ of minimal degree. Hence, $\deg(\phi)=\deg(\phi_*)$ and Theorem \ref{univbound} shows that
$$
\ga=\mu'(\phi)=\mu'(\phi_*)=\mu'_*(\phi_*)=\ga_*.
$$
In particular, $\mu(\phi_*-\phi)=\mu'(\phi_*-\phi)\ge \ga$.
\end{proof}



\subsection{Depth zero valuations}
Take $a\in K$, $\ga\in\Lambda\infty$.
The following valuation $\mu=\omega_{a,\ga}$ on $\kx$ is said to be a \emph{depth zero valuation}:
$$
\mu\left(\sum\nolimits_{0\le s}a_s(x-a)^s\right) = \min\left\{v(a_s)+s\ga\mid0\le s\right\}.
$$

If $\ga<\infty$, then $x-a$ is a key polynomial for $\mu$; hence, $\deg(\mu)=1$. 

It is easy to check that
\begin{equation}\label{mu0eq}
\as{1.5}
\begin{array}{c}
\omega_{a,\ga}\le\omega_{b,\delta}\sii \ga\le\delta\quad\mbox{and}\quad v(a-b)\ge \ga,\\
(-\infty,\omega_{a,\ga})_{\La}=\left\{\omega_{a,\delta}\mid \delta\in\La,\ \delta<\ga\right\}.
\end{array}
\end{equation}

\subsection{Degree of maximal valuations}\label{subsecDegMax}
Let $\mu$ be a valuation on $\kx$, and let $\gm\hookrightarrow\La$ be an embedding of $\gm$ into a dense ordered group. 
That is, for any $\be<\ga$ in $\La$, there exists $\delta\in\La$ such that $\be<\delta<\ga$. For instance, $\La=\left(\gm\right)_\Q$ satisfies this condition.

By Proposition \ref{gap}, any valuation $\rho<\mu$ admits key polynomials and $\deg(\rho)$ is defined. For any open interval $I$ of valuations we may consider its \emph{set of degrees}:
$$
\dg\left(I\right)=\{\deg(\rho)\mid \rho\in I\}\subset\N.
$$


For instance, if $\mu_0=\omega_{a,\ga}$ is a depth zero valuation, we deduce from 
(\ref{mu0eq}) that $$\dg\left((-\infty,\mu_0)_{\La}\right)=\{1\}.$$

Also, for an augmented valuation $\mu'=[\mu;\,\phi,\ga]$, Lemma \ref{intervalOrd} shows that $$\dg\left((\mu,\mu')_{\La}\right)=\{\deg(\phi)\}.$$ 

If $\mu$ is a maximal valuation, we may define
$$
\deg(\mu):=\sup\left(\dg\left((-\infty,\mu)_{\left(\gm\right)_\Q}\right)\right)\in \N\infty
$$

We shall see in section \ref{secMLV} that there are maximal valuations of infinite degree.


\section{Limit augmentations}\label{secLimAug}

Consider two valuations $\mu,\,\nu$ with values in a common ordered group, such that $\mu<\nu$. An iterative application of Proposition \ref{gap}, yields a chain of ordinary augmentations
$$
\mu\ <\ \mu_1\ <\ \cdots\ <\ \mu_n\ <\ \cdots \ \le\ \nu.
$$
each one getting closer to $\nu$ than the previous one. 
Unfortunately, this chain does not always ``converge" to $\nu$.

In order to overcome this obstacle, Vaqui\'e introduced \emph{limit augmentations} based on certain \emph{limit key polynomials}.

This section is devoted to discuss these limit augmentations.

\subsection{Continuous families of augmentations}\label{subsecCFA}
Let $A$ be a totally ordered set.

A  \emph{totally ordered family of valuations} parameterized by  $A$ is a totally ordered set
  $$\aa=\cfa$$ of valuations taking values in a common ordered group, such that
 the bijection $i\to\rho_i$ is an isomorphism of ordered sets between $A$ and $\aa$.

A polynomial $f\in\kx$ is \emph{$\aa$-stable} if  for some index $i_0\in A$ we have 
$$
\rho_{i}(f)=\rho_{i_0}(f), \quad\ \forall\,i\ge i_0.
$$
This stable value is denoted $\rhi(f)$. 

We obtain in this way a \emph{stability function} $\rhi$ which is defined only on the multiplicatively closed subset of $\kx$ formed by the stable polynomials.

By Corollary \ref{Phiclass} (2), the instability of  $f\in\kx$ is characterised as follows: 
\begin{equation}\label{unstable}
f\quad\mbox{unstable}\quad \sii \quad \rho_i(f)<\rho_j(f),\qquad \forall\,i<j.
\end{equation}

We say that $\aa$ has a \emph{stable limit} if all polynomials in $\kx$ are stable. In this case, $\rhi$ is a valuation on $\kx$, which is called the stable limit of $\aa$. We write $\rhi=\lim_{i\in A}\rho_i$.

\begin{proposition}\label{stablelimit}
Let $\aa=\cfa$ be a totally ordered family of valuations having a stable limit. If the set $A$ contains no maximal element, the valuation $\rhi$ has trivial support and satisfies $\kp(\rhi)=\emptyset$. In particular, $\rhi$ is a maximal valuation.
\end{proposition}

\begin{proof}
Since $A$ has no maximal element, we have $\rho_i<\rhi$ for all $i\in A$. By Theorem \ref{maximal}, all the valuations $\rho_i$ have trivial support. 
All non-zero $f\in\kx$ satisfy $\rhi(f)=\rho_i(f)<\infty$ for some $i\in A$. Therefore, $\rhi$ has trivial support too. 

All non-zero homogeneous elements in $\gg_{\rhi}$ are units by Corollary \ref{KerImGr}. By Theorem \ref{empty}, $\kp(\rhi)=\emptyset$ and $\rhi$ is maximal by Theorem \ref{maximal}.
\end{proof}\e

Let us focus our attention on totally ordered families of valuations of a special type.

\begin{definition}\label{defcont}
Let $\mu$ be a non-maximal valuation and let $\gm\hookrightarrow \La$ be an embedding of ordered groups.
Consider a family  of ordinary augmentations of $\mu$, 
$$\aa=\cfa,\qquad \rho_i=[\mu; \chi_i,\be_i],\qquad \be_i\in\La,\ \be_i>\mu(\chi_i),$$
parameterized by a totally ordered set $A$ such that $\be_i<\be_j$ for all $i<j$ in $A$.

We say that  $\aa$ is a \emph{continuous family of augmentations} of $\mu$ of stable degree $m$ it it satisfies the following conditions: 
 \begin{enumerate}
\item The set $A$ contains no maximal element.
\item All  key polynomials $\chi_i\in\kpm$ have degree $m$.
\item For all $i<j$ in $A$, $\chi_j$ is a  key polynomial for $\rho_i$ and satisfies
$$
\chi_j\not\sim_{\rho_i}\chi_i\quad\mbox{ and }\quad \rho_j=[\rho_i;\chi_j,\be_j].
$$
\end{enumerate}
\end{definition}

The basic examples of continuous families of augmentations are provided by certain valuations $\nu$ on $\kx$ such that $\mu<\nu$.

We denote the common degree of all polynomials in the set $\phmn$ by $\deg\left(\phmn\right)$.

\begin{proposition}\label{nuIscont}
Let $\nu$ be a valuation on $\kx$ such that $\mu<\nu$. Suppose that the set $A=\nu\left(\phmn\right)$ does not contain a maximal element in $\left(\gn\right)\!\infty$. For all $\al\in A$, choose some polynomial $\chi_\al\in\Phi_{\mu,\nu}$ such that $\nu(\chi_\al)=\al$, and build $\rho_\al=[\mu; \chi_\al,\al]$. Then, $\aa=\left(\rho_\al\right)_{\al\in A}$ is a continuous family of augmentations of $\mu$ of stable degree $m=\deg\left(\phmn\right)$. 

Moreover, all polynomials in $\kx$ of degree $m$ are stable.
\end{proposition}

\begin{proof}
Clearly, the  family $\aa$ satisfies conditions (1) and (2) of Definition \ref{defcont}.

For all $a\in\kx$ with $\deg(a)<m$, we have $\mu(a)=\nu(a)$ by the definition of $\phmn$, and $\mu(a)=\rho_\al(a)$ for all $\al\in A$, by the definition of the augmented valuation $\rho_\al$.

For $\al<\be$ in $A$, write $\chi_\be=\chi_\al+a$ for some $a\in\kx$ of degree less than $m$.
Since $\nu(\chi_\al)=\al<\be=\nu(\chi_\be)$, we deduce that $\mu(a)=\nu(a)=\al$. By the definition of the augmented valuations, 
$$
\rho_\al(\chi_\be)=\al<\be=\rho_\be(\chi_\be),\qquad 
\rho_\al(\chi_\al)=\al=\rho_\be(\chi_\al).
$$ 
By Corollary \ref{Phiclass}, $\chi_\al\not\in\Phi_{\rho_\al,\rho_\be}=[\chi_\be]_{\rho_\al}$. Thus, $\chi_\be$ is a  key polynomial for $\rho_\al$ and $\chi_\al\not\sim_{\rho_\al}\chi_\be$.

Also, $[\rho_\al;\,\chi_\be,\be]=\rho_\be$, because both valuations coincide on $\chi_\be$-expansions. This proves that $\aa$ satisfies (3).

Finally, let us see that all monic $f\in\kx$ of degree $m$ are stable.
 Suppose that $f$ is unstable, and write $f=\chi_\be+a_\be$ for all $\be\in A$, where $a_\be\in\kx$ has degree less than $m$. Then, for all $\al<\be\in A$, we have 
\begin{equation}\label{nnst}
 \mu(f)\le\rho_\al(f)<\rho_\be(f)=\min\{\be,\nu(a_\be)\}\le\nu(f).
\end{equation}
This implies $f\in\Phi_{\rho_\al,\nu}\subset\phmn$; in particular, $\nu(f)\in A$. On the other hand, $\Phi_{\rho_\al,\nu}=[\chi_\be]_{\rho_\al}$, so that $f\sim_{\rho_\al}\chi_\be$ and $\rho_\al(f)= \rho_\al(\chi_\be)=\al$. We deduce from (\ref{nnst}) that $\al<\nu(f)$ for all $\al\in A$. This is impossible, because $\nu(f)\in A$.   
\end{proof}\e

The following properties hold for all continuous families of augmentations, and they are easily deduced from the definitions. 
\begin{itemize}
\item $\aa$ is a totally ordered family of valuations, parameterized by the set $A$. 
\item All polynomials in $\kx$ of degree less than $m$ are stable. 
\item For all $i,j\in A$, $\rho_i(\chi_j)=\min\{\be_i,\be_j\}$.  
\item For all $i\in A$, $\deg(\rho_i)=m$ and $\chi_i$ is a proper key polynomial for $\rho_i$. By Theorem \ref{incomm}, $\rho_i$ is residually transcendental. 
\item For all $i<j$ in A, $\Phi_{\rho_i,\rho_j}=[\chi_j]_{\rho_i}$ and this class is proper. 
\end{itemize}

By the definition of the augmented valuation, $\g_{\rho_i,m}=\g_{\rho_j,m}$ for all $i<j$. Thus, Lemma \ref{groupchain0} shows that
\begin{equation}\label{samegroup}
\g_{\rho_i}=\g_{\rho_i,m}=\g_{\rho_j,m}=\g_{\rho_j}.
\end{equation}
In particular, $\rho_i$ has relative ramification index equal to one.

The common value grup $\ginf:=\g_{\rho_i}$ for all $i\in A$, is called the \emph{stable value group} of the continuous family. Note that $\be_i\in\ginf$ for all $i$. 

The class $\Phi_{\mu,\aa}:=\Phi_{\mu,\rho_i}=[\chi_i]_\mu$ does not depend on $i$, by Corollary \ref{Phiclass} (2). The following result is a consequence of Lemma \ref{groupchain0}, by the same argument as in (\ref{samegroup}).

\begin{lemma}\label{mu=rinf}
 If the class $\Phi_{\mu,\aa}$ is proper, then $\gm=\ginf$.
\end{lemma}

\noindent{\bf Remark. }Any cofinal family of $\left(\rho_i\right)_{i\in A}$ will have the same limit behaviour. Since all totally ordered sets admit well-ordered cofinal subsets, we may always assume that the set $A$ is well-ordered.

\subsection*{Essential continuous families of augmentations}
Let $\mi$ be the minimal degree of an unstable   polynomial. We agree that $\mi=\infty$ if all polynomials are stable.

Any continuous family falls in one of the following three cases:

\begin{enumerate}
\item[(a)] It has a \emph{stable limit}.
That is, $\mi=\infty$, so that  the function $\rhi$ is a valuation on $\kx$. This valuation is commensurable and satisfies $\kp(\rhi)=\emptyset$.\e

\item[(b)] It is \emph{inessential}.
That is, $\mi=m$.\e

\item[(c)] It is \emph{essential}.
That is, $m<\mi<\infty$.\e
\end{enumerate}

Let $\nu$ be a valuation on $\kx$ such that $\rho_i<\nu$ for all $i\in A$.

If $\left(\rho_i\right)_{i\in A}$ is inessential and $f\in\kx$ is an unstable   polynomial of degree $m$, then the ordinary augmentation $\mu'=[\mu; f,\nu(f)]$ satisfies 
$$
\rho_i<\mu'\le\nu,\qquad \forall\,i\in A.
$$
In other words, $\mu'$ is closer to $\nu$ than any $\rho_i$, and $\mu'$ is obtained from $\mu$ by a single augmentation. 
This justifies why we call it ``inessential".

\subsection{Limit key polynomials}
Let us fix  an essential continuous family $\aa=(\rho_i)_{i\in A}$ of augmentations of $\mu$ of stable degree $m$.

We define the set $\kpi\left(\aa\right)$ of \emph{limit key polynomials} for $\aa$
as the set of monic unstable polynomials in $\kx$ of minimal degree $\mi$.
We recall that $\mi>m$.

Let us fix a limit key polynomial $\phi\in\kpi\left(\aa\right)$.
Since the product of stable polynomials is stable, $\phi$ is irreducible in $\kx$.


Let $\ginf\hookrightarrow\La$ be an embedding of ordered groups, and choose $\ga\in\La\infty$ such that $\rho_i(\phi)<\ga$ for all $i\in A$.
We denote by  $\mu'=[\aa;\phi,\ga]$ the following mapping: 
$$\mu':\kx\longrightarrow \La\infty,\qquad f=\sum\nolimits_{0\le s}a_s\phi^s\longmapsto \mu'(f)=\min\{\rhi(a_s)+s\ga\mid 0\le s\},
$$
defined in terms of $\phi$-expansions. Note that $\mu'(\phi)=\ga$.

This \emph{limit augmentation} shares some features with the ordinary augmentations. 

\begin{proposition}\cite[Sec. 1.4]{Vaq},\cite[Sec. 7]{KP}\label{extensionlim} 
\mbox{\null}
\begin{enumerate}
\item The mapping $\mu'=[\aa;\phi,\ga]$ is a valuation on $\kx$. 

If $\ga<\infty$, it has trivial support. If $\ga=\infty$, the support of $\mu'$ is $\phi\kx$.  
\item For all $i\in A$ we have $\rho_i< \mu'$. Moreover, for all non-zero $f\in\kx$,  
$$f\mbox{ is stable} \sii \exists\, i\in A \mbox{ such that }\rho_i(f)=\mu'(f).$$
In this case,  $\inmp f$ is a unit in $\ggmp$.
\item If $\ga<\infty$, then $\phi$ is a key polynomial for $\mu'$, of minimal degree. In particular, $\deg(\mu')=\deg(\phi)=\mi$.
\item Suppose that the class $\Phi_{\mu,\aa}$ is proper. Then,
$$ \gmp=\gen{\gm,\ga},\ \mbox{ if }\ \ga<\infty;\qquad \gmp=\gm,\ \mbox{ if }\ \ga=\infty.
$$ 
\end{enumerate}
\end{proposition}

\noindent{\bf Example. }Let $p$ be a prime number and consider the $p$-adic valuation $v=\ord_p$ on $\Q$. Take a $p$-adic number with non-zero $p$-adic coefficients
$$
\t=\sum_{i=0}^\infty c_ip^i\in\Z_p,\qquad c_i\in\{1,\dots,p-1\}, 
$$
and denote its partial sums by $a_i=\sum_{0\le j<i}c_ip^i\in\Z$, for all $i>0$.

 Consider the depth zero valuation $\mu=\omega_{0,0}$ and the 
chain of ordinary augmentations
$$
\mu\ \stackrel{x-a_1,1}\lra\  \rho_1\ \stackrel{x-a_2,2}\lra\ \cdots
\ \lra\ \rho_{i-1} 
\ \stackrel{x-a_{i},i}\lra\ \rho_{i} 
\ \lra\ \cdots
$$
All these valuations have depth zero and are augmentations of $\mu$: $$\rho_i=\omega_{a_i,i}=[\mu;\,x-a_i,i].$$ They form a continuous family $\aa=\left(\rho_i\right)_{i\ge1}$ of augmentations of $\mu$. One checks easily that a non-zero $f\in\Q[x]$ is stable if and only if $f(\t)\ne0$.

Let us analyze the different possibilities for the limit of this family.\e

$\bullet$ \  If $\t\in\Q$, then $x-\t$ is an unstable polynomial of degree one. Thus, the family is inessential. The natural limit valuation $\mu'=\omega_{\t,\infty}$ has depth zero.\e

$\bullet$ \ If $\t$ is transcendental over $\Q$, then the family $\aa$ has a stable limit $\mu'=\rhi$.\e

With the notation of the introduction, $\mu'$ has quasi-finite depth zero. 

$\bullet$ \ If $\t\in\overline{\Q}\setminus\Q$, the family is essential and the minimal polynomial $\phi$ of $\t$ over $\Q$ is a limit key polynomial for $\aa$. Consider the limit augmentation $\mu'=[\aa;\,\phi,\infty]$.

With the notation of the introduction, $\mu'$ has finite depth equal to one.




\subsection{Unicity of limit augmentations}

\begin{lemma}\label{critEquiv}
Let $\mu$, $\mu^*$ be two valuations on $\kx$ extending $v$. Suppose they admit essential continuous families of augmentations 
\begin{equation}\label{rhorho}
\aa=\cfa,\quad \rho_i=[\mu;\,\chi_{i},\be_{i}]; \qquad \aa^*=\left(\rho^*_j\right)_{j\in A^*},\quad \rho_j^*=[\mu^*;\,\chi^*_{j},\be^*_{j}].
\end{equation} 
of stable degrees $m$, $m^*$, respectively. 
Consider embeddings $\ginf\hookrightarrow\Lambda$, $\g_{\aa^*}\hookrightarrow\Lambda$ into some common ordered group, whose restrictions to $\g$ coincide.

Then, the following conditions are equivalent.
\begin{enumerate}
\item The sets $\aa$ and $\aa^*$ are one cofinal in each other, with respect to the partial ordering $\le$ of valuations taking values in $\Lambda$. 
\item $m=m^*$ and  $\rhi=\rho_{\aa^*}$.
\item $m=m^*$ and there is a  valuation $\nu$ such that $\rho_i,\rho^*_j<\nu$ for all $i\in A$, $j\in A^*$.
\end{enumerate}
If these conditions hold, we say that the continuous families $\aa$ and $\aa^*$ are equivalent.
\end{lemma}

\begin{proof}
Suppose that $\aa$ and $\aa^*$ are one cofinal in each other. By the criterion of equation (\ref{unstable}), the two families have the same stable polynomials and $\rhi=\rho_{\aa^*}$. 

Also, there exist  $i,k\in A$, $j\in A^*$ such that $\rho_i<\rho^*_j<\rho_k$. By Corollary \ref{KerImGr} (3), the stable degrees coincide: $m=\deg(\rho_i)=\deg(\rho^*_j)=m^*$. Thus, (1) implies (2).

Now, suppose $m=m^*$ and $\rhi=\rho_{\aa^*}$.  That is, $\aa$ and $\aa^*$ have the same stable polynomials, and the stable values of these polynomials coincide. 

If $\aa$ has a stable limit, then $\nu=\rhi=\rho_{\aa^*}$ satisfies condition (3). 

Otherwise, for any $\phi\in\kpi(\aa)=\kpi(\aa^*)$ the limit augmentation valuation 
$$\nu=[\aa;\phi,\infty]=[\aa^*;\phi,\infty]$$ is a common upper bound of $\aa$ and $\aa^*$. Thus, (2) implies (3).

Condition (3), implies that the set $\aa\cup\aa^*$ 
is totally ordered, by Theorem \ref{interval}. Now, suppose that $\aa$ is not cofinal in $\aa^*$. There exists an index $j\in A^*$ such that $\rho^*_j>\rho_i$ for all $i\in A$. For any $k>j$ in $A^*$, we have 
$$\rho_j^*(\chi_k^*)=\be_j^*< \be_k^*=\rho_k^*(\chi_k^*).$$
This implies that $\chi^*_k$ is not $\aa$-estable. Since $\deg(\chi_k^*)=m^*=m$, this contradicts the fact that $\aa$ is essential. Thus, (3) implies (1).
\end{proof}\e


Let us check when different building data $\phi$, $\ga$ lead to the same limit augmentation.

\begin{lemma}\label{unicitylim}
Let $\aa$ and $\aa^*$ be essential continuous families of augmentations as in (\ref{rhorho}). Take $\phi\in\kpi\left(\aa\right)$,  $\phi^*\in\kpi\left(\aa^*\right)$. Choose $\ga,\ga^*\in\Lambda\infty$ such that $\rho_i(\phi)<\ga$ for all $i\in A$, and  $\rho^*_j(\phi^*)<\ga^*$ for all $j\in A^*$. 

Then, $\mu'=[\aa;\,\phi,\ga]$ coincides with $(\mu')^*=[\aa^*;\,\phi^*,\ga^*]$ if and only if 
\begin{equation}\label{equalmulim}
\rhi=\rho_{\aa^*},\qquad\deg(\phi)=\deg(\phi^*)\quad\mbox{and}\quad\rhi(\phi^*-\phi)\ge \ga=\ga^*.
\end{equation}
\end{lemma}

\begin{proof}
If the conditions of (\ref{equalmulim}) hold, we deduce $\mu'=(\mu')^*$ by a completely analogous argument to that used in the proof of Lemma \ref{unicityOrd}. 

Conversely, suppose $\mu'=(\mu')^*$. 
By Proposition \ref{extensionlim}, $\phi$, $\phi^*$ are key polynomials for $\mu'$ of minimal degree. Hence, $\deg(\phi)=\deg(\phi^*)$ and Theorem \ref{univbound} shows that 
$$
\ga=\mu'(\phi)=\mu'(\phi^*)=(\mu')^*(\phi^*)=\ga^*.
$$
In particular, $\rhi(\phi_*-\phi)=\mu'(\phi_*-\phi)\ge \ga$. It remains only to show that $\rhi=\rho_{\aa^*}$.

 By Lemma \ref{critEquiv}, it suffices to see that $\aa$, $\aa^*$ are cofinal one into each other.
Since $\rho_i,\rho^*_j<\mu'$ for all $i\in A$, $j\in A^*$,  the set $\aa\cup\aa^*$ is totally ordered by Theorem \ref{interval}. 

Suppose that  $\rho^*_j>\rho_i$ for all $i\in A$, for some index $j\in A^*$. Arguing as in the proof of Lemma \ref{critEquiv}, we deduce that for any $k>j$ in $A^*$, the key polynomial $\chi^*_k$ is not $\aa$-stable. Since $m^*=\deg(\chi^*_{k})<\deg(\phi^*)= \deg(\phi)$, this contradicts the minimality of $\deg(\phi)$ among all unstable polynomials. 
\end{proof}

\begin{lemma}\label{intervalLim}
Let $\aa=\cfa$, with $\rho_i=[\mu;\,\chi_{i},\be_{i}]$, be a continuous family of augmentations of stable degree $m$, of some valuation $\mu$. Let $\ginf\hookrightarrow\La$ be an embedding of ordered groups.

Since all key polynomials $\chi_i$ are $\mu$-equivalent, the value $\be_{\op{min}}:=\mu(\chi_i)$ is independent of $i\in A$.
 Let $S$ be the initial segment of $\La_{>\be_{\op{min}}}$ generated by the set $\{\be_i\mid i\in A\}$.
For all $\be\in S$, let $\rho_\be=[\mu;\chi_{i},\be]$, for some $i\in A$ such that $\be_i>\be$. Then,\e

(1) \ If $\aa$ has a stable limit, then $\,(\mu,\rhi)_{\Lambda}=\left\{\rho_\be\mid\be\in S\right\}$, $\dg\left((\mu,\rhi)_{\Lambda}\right)=\{m\}$.\e

(2) \ If $\aa$ is essential, let $\mu'=[\aa;\phi,\ga]$ be a  limit augmentation for some $\phi\in\kpi(\aa)$ and $\ga\in \La\infty$ such that $\ga>T:=\left\{\rho_j(\phi)\mid j\in A\right\}$. Then,
$$
(\mu,\mu')_{\La}=\left\{\rho_\be\mid\be\in S\right\}\cup \left\{\mu_{\delta}\mid\delta\in\La,\ T<\delta<\ga\right\},
$$
where $\mu_{\delta}=[\aa;\phi,\delta]$. 
In particular, $\dg\left((\mu,\mu')_{\Lambda}\right)=\{m,\,\deg(\phi)\}$.
\end{lemma}

\begin{proof}
The valuation $\rho_\be=[\mu;\chi_{i},\be]$ does not depend on $i$, by Lemma \ref{unicityOrd}. 

If $\aa$ has a stable limit, let us denote $\mu'=\rhi$ as well.

Let $\rho$ be a $\La$-valued valuation such that $\mu<\rho<\mu'$. 
If $\rho\le\rho_{i}$ for some $i\in A$, then $\rho=\rho_\be$ for some $\be\in S$, by Lemma \ref{intervalOrd}. 

Suppose that $\rho>\rho_i$ for all $i\in A$. If $\aa$ has a stable limit, then $\rho=\rhi$ by Proposition \ref{stablelimit}, against our assumption that $\rho<\mu'$. This proves (1).

If $\aa$ has no stable limit, then $\rho$ coincides with $\mu'$ on stable polynomials by Corollary \ref{Phiclass}. 
If $\rho(\phi)\ge \ga=\mu'(\phi)$, then  the action of both valuations on $\phi$-expansions shows that $\rho\ge\mu'$, against our assumption. Thus, $\rho(\phi)< \ga=\mu'(\phi)$.

This implies that $\phi$ is a key polynomial for $\rho$, by Proposition \ref{gap}.  Therefore, for $\delta=\rho(\phi)$, we get $\rho=\mu_{\delta}$, because both valuations coincide on $\phi$-expansions. 
\end{proof}

\section{MacLane-Vaqui\'e chains}\label{secMLV}

Consider a finite, or countably infinite, chain of mixed augmentations 
\begin{equation}\label{depthMLV}
\mu_0\ \stackrel{\phi_1,\ga_1}\lra\  \mu_1\ \stackrel{\phi_2,\ga_2}\lra\ \cdots
\ \lra\ \mu_{n-1} 
\ \stackrel{\phi_{n},\ga_{n}}\lra\ \mu_{n} \ \lra\ \cdots
\end{equation}
in which every node is an augmentation  of the previous node, of one of the following types:\e

\emph{Ordinary augmentation}: \ $\mu_{n+1}=[\mu_n;\, \phi_{n+1},\ga_{n+1}]$, for some $\phi_{n+1}\in\kp(\mu_n)$.\e

\emph{Limit augmentation}:  \ $\mu_{n+1}=[\aa;\, \phi_{n+1},\ga_{n+1}]$,  for some $\phi_{n+1}\in\kpi(\aa)$, where $\aa$ is an essential continuous family of augmentations of $\mu_n$.\e




Let us fix $\phi_0\in\kp(\mu_0)$ a key polynomial of minimal degree, and let $\ga_0=\mu_0(\phi_0)$. 

The values $\ga_n$ belong to $\La\infty$ for some ordered group $\La$. For all $n\ge0$, we shall refer to $(\phi_n,\ga_n)$ as the \emph{building pair} of $\mu_n$.

The following properties of an infinite chain of mixed augmentations follow from Theorem \ref{charKP}, Propositions \ref{extension}, \ref{gap}, \ref{extensionlim} and Corollary \ref{Phiclass}.\e

$\bullet$ \ \,$\ga_n=\mu_n(\phi_n)<\ga_{n+1}$.\e

$\bullet$ \ \,$\phi_n$ is a key polynomial for $\mu_n$ of minimal degree; thus,$$\deg(\mu_n)=\deg(\phi_n)\mid \deg\left(\Phi_{\mu_n,\mu_{n+1}}\right).$$
$$
 \bullet \quad   \Phi_{\mu_n,\mu_{n+1}}=
 \begin{cases}
 [\phi_{n+1}]_{\mu_n},&\mbox{ if \ $\mu_n\to\mu_{n+1}$  \ ordinary augmentation},\\
 \Phi_{\mu_n,\aa}=[\chi_i]_{\mu_n},\  \forall\, i\in A,&\mbox{ if \ $\mu_n\to\mu_{n+1}$  \ limit augmentation}.
 \end{cases}\quad\ 
$$
$$
    \deg\left(\Phi_{\mu_n,\mu_{n+1}}\right)=
 \begin{cases}
\deg(\phi_{n+1}),&\mbox{if \ $\mu_n\to\mu_{n+1}$  \ ordinary augmentation},\\
 \mbox{stable degree of }\aa,&\mbox{if \ $\mu_n\to\mu_{n+1}$  \ limit augmentation}.
 \end{cases}
$$

\subsection{Definition of MacLane-Vaqui\'e chains}\label{subsecDefMLV}
Since we want our chains of mixed augmentations  to be as unique as possible, we shall impose a technical condition.

\begin{definition}\label{defMLV}
A finite, or  countably infinite, chain of mixed  augmentations as in {\rm (\ref{depthMLV})} is a \emph{MacLane-Vaqui\'e (abbreviated MLV) chain}  if every augmentation step satisfies:
\begin{itemize}
\item If $\,\mu_n\to\mu_{n+1}\,$ is ordinary, then $\ \deg(\mu_n)<\deg\left(\Phi_{\mu_n,\mu_{n+1}}\right)$.
\item If $\,\mu_n\to\mu_{n+1}\,$ is limit, then $\ \deg(\mu_n)=\deg\left(\Phi_{\mu_n,\mu_{n+1}}\right)$ and $\ \phi_n\not \in\Phi_{\mu_n,\mu_{n+1}}$. 
\end{itemize}

A \mlv chain is \emph{complete} if the valuation  $\mu_0$ has depth zero.
\end{definition}

Let us point out some specific features of complete infinite MLV-chains. Denote
$$m_n=\deg(\mu_n)=\deg(\phi_n)\quad \mbox{for all}\quad n\ge0.$$ 

$\bullet$ \ $m_n<m_{n+1}$. \e

If  $\mu_n\to \mu_{n+1}$ is an ordinary augmentation, then $m_n<\deg(\phi_{n+1})$ by the MLV condition.

If  $\mu_n\to \mu_{n+1}$ is a limit augmentation, then $m_n$ is the stable degree of the essential continuous family $\aa$ of augmentations of $\mu_n$. Since $\phi_{n+1}\in\kpi(\aa)$, we have $\deg(\phi_{n+1})>m_n$.\e

$\bullet$ \ $\Phi_{\mu_n,\mu_{n+1}}$ is a proper class of key polynomials for $\mu$. \e

Indeed, $\phi_n$ is a key polynomial for $\mu_n$ of minimal degree, and $\phi_n\not\in\Phi_{\mu_n,\mu_{n+1}}$.\e

$\bullet$ \ $\mu_n$ is residually transcendental  and $\ga_n\in \gq$.\e 

Since $\mu_n$ admits proper classes of key polynomials, it is commensurable over $v$ by Theorem \ref{incomm}. 

However, in a finite MLV-chain of length $r$, the value $\ga_r$ may  not belong to $\gq$, or be equal to $\infty$. Thus, $\mu_r$ may be incommensurable, or have non-trivial support.\e

$\bullet$ \ Every infinite MacLane-Vaqui\'e chain has a stable limit.  \e

Indeed, any $f\in\kx$ has $\deg(f)<m_n$ for some $n$. Hence, $\mu_{n}(f)=\mu_{n+1}(f)$ by Proposition \ref{gap}, 
since $m_n\le\deg\left(\Phi_{\mu_{n},\mu_{n+1}}\right)$ for both types of the augmentation $\mu_{n}\to\mu_{n+1}$.

The criterion of equation (\ref{unstable}) shows that $f$ is stable.\e



$\bullet$ \ By Propositions \ref{extension} and \ref{extensionlim}, we get a chain of value groups:
\begin{equation}\label{chainGamma}
\g_{\mu_{-1}}:=\g\subset\g_{\mu_0}\subset \cdots \subset \g_{\mu_{n}}\subset \cdots\qquad \mbox{ with }\quad\g_{\mu_{n}}=\gen{\g_{\mu_{n-1}},\ga_{n}},\ \forall\,n\ge 0.
\end{equation}

$\bullet$ \  $\g_{\mu_n,m_n}=\g_{\mu_{n-1}}$. In particular, $\ e_n:=\erel(\mu_n)=\left(\g_{\mu_n}\colon \g_{\mu_{n-1}}\right)$.\e

Indeed, take any $a\in\kx$ such that $\deg(a)<m_n$. 

If $\mu_{n-1}\to\mu_n$ is an ordinary augmentation, then $\mu_{n-1}(a)=\mu_n(a)$, because $m_n=\deg\left(\Phi_{\mu_{n-1},\mu_{n}}\right)$. Thus,  $\g_{\mu_{n},m_n}\subset\g_{\mu_{n-1}}$.
On the other hand, $\phi_n$ is a proper key polynomial for $\mu_{n-1}$ and Lemma \ref{groupchain0} shows that $\g_{\mu_{n-1}}=\g_{\mu_{n-1},m_n}\subset\g_{\mu_n,m_n}$. 

If $\mu_n=[\aa; \phi_n,\ga_n]$ is a limit augmentation, then $a$ is $\aa$-stable, so that $\mu_n(a)=\rhi(a)$. Thus,  $\g_{\mu_{n},m_n}\subset\ginf=\g_{\mu_{n-1}}$, the last equality by Lemma \ref{mu=rinf}.

Also, the proper class $\Phi_{\mu_{n-1},\mu_n}$ has degree $m_{n-1}<m_n$. By Lemma \ref{groupchain0}, $$\g_{\mu_{n-1}}=\g_{\mu_{n-1},m_{n-1}}=\g_{\mu_n,m_{n-1}}\subset\g_{\mu_n,m_n}.$$



\subsection{The main theorem of MacLane-Vaqui\'e}



\begin{lemma}\label{singlestep}
Let $(\mu_n)_{0\le n\le r}$ be a complete finite MLV-chain. Let $\mu_r \stackrel{\phi,\ga}\lra \mu$ be an ordinary, or limit augmentation. Then, $\mu$ is the last valuation of a complete finite MLV-chain, keeping $(\phi,\ga)$ as the building pair of $\mu$ in the last augmentation step.
\end{lemma}



\begin{proof}
If the augmentation step $\mu_r\to\mu$ satisfies the MLV condition of Definition \ref{defMLV}, the statement  is obvious. Let us suppose that this is not the case. \e

\noindent{\bf Case 1.} The augmentation $\mu=[\mu_r;\,\phi,\ga]$ is ordinary and $\deg(\mu_r)=\deg(\phi)$.

If $r=0$, then $\phi=x-a$ for some $a\in K$ and $\mu$ is a depth zero valuation $\mu=\omega_{a,\ga}$. Thus, $\mu$ forms a complete MLV-chain of length zero whose building pair is $(\phi,\ga)$.

If $r>0$ and $\mu_r=[\mu_{r-1};\,\phi_r,\ga_r]$ is an ordinary augmentation, then $$\deg(\mu_{r-1})<\deg(\phi_r)=\deg(\mu_r)=\deg(\phi).$$ By Corollary \ref{Phiclass},
$$
\phi\in\Phi_{\mu_r,\mu}\subset\Phi_{\mu_{r-1},\mu}=\Phi_{\mu_{r-1},\mu_r}=[\phi_r]_{\mu_{r-1}}.
$$
Hence, $\phi$ is a key polynomial for $\mu_{r-1}$ and the ordinary augmentation $[\mu_{r-1};\,\phi,\ga]$ satisfies the MLV condition. Also, $[\mu_{r-1};\,\phi,\ga]=\mu$ because both valuations coincide on $\phi$-expansions.   

If $r>0$ and  $\mu_r=[\aa;\,\phi_r,\ga_r]$ is a limit augmentation, then 
$$\deg(\mu_{r-1})=\deg\left(\Phi_{\mu_{r-1},\mu_r}\right)=\mbox{ stable degree of }\aa<\deg(\phi_r)=\deg(\phi).$$

Since $\phi$ is a key polynomial for $\mu_r$, $\op{in}_{\mu_r}\phi$ is not a unit. By Proposition \ref{extensionlim} (2), $\phi$ is unstable of minimal degree for the family $\aa$; thus, $\phi\in\kpi(\aa)$. As before,  $\mu=[\aa;\,\phi,\ga]$ and this  limit augmentation satisfies the MLV condition, because the essential continuous family $\aa$ of augmentations of $\mu_{r-1}$ remains unchanged.\e


\noindent{\bf Case 2.} The augmentation $\mu=[\aa;\,\phi,\ga]$ is a limit augmentation with respect to some essential continuous family $\aa=\cfa$ of augmentations of $\mu_r$. Let $\rho_i=[\mu_r;\,\chi_{i},\be_{i}]$. 

Take any $i\in A$.
We may decompose the augmentation step $\mu_r\to\mu$ into two pieces: 
$$
\mu_r\ \stackrel{\chi_i,\be_i}\lra\ \rho_i\ \stackrel{\phi,\ga}\lra\ \mu,
$$
where $\mu_r\to\rho_i$ is an ordinary augmentation, and $\mu=[\aa_{>i};\,\phi,\ga]$ is a limit augmentation with respect to the essential continuous family $\aa_{>i}$ of augmentations of $\rho_i$.

By the previous case,  $\rho_i$ is the last valuation of a complete finite MLV-chain keeping $(\chi_i,\be_i)$ as the building pair of $\rho_i$.
Also, the limit augmentation $\rho_i\to\mu$ satisfies the MLV condition. Indeed,  the class 
$$\Phi_{\rho_i,\mu}=\Phi_{\rho_i,\aa_{>i}}=[\chi_j]_{\rho_i}\quad\mbox{for all}\quad j>i,$$
has degree $\deg(\chi_j)=\deg(\rho_i)$, and $\chi_i\not\in \Phi_{\rho_i,\mu}$ because  $\chi_i\not\sim_{\rho_i}\chi_j$ for all $j>i$, by condition (3) of Definition \ref{defcont}.
\end{proof}

\begin{theorem}\label{main}
Every valuation $\nu$ on $\kx$ falls in one of the following cases.  \e

(a) \ It is the last valuation of a complete finite MacLane-Vaqui\'e chain.
$$ \mu_0\ \stackrel{\phi_1,\ga_1}\lra\  \mu_1\ \stackrel{\phi_2,\ga_2}\lra\ \cdots\ \lra\ \mu_{r-1}\ \stackrel{\phi_{r},\ga_{r}}\lra\ \mu_{r}=\nu.$$

(b) \ It is the stable limit of a continuous family $\aa=\cfa$ of augmentations of some valuation $\mu_r$ falling in case (a):
$$ \mu_0\ \stackrel{\phi_1,\ga_1}\lra\ \mu_1\ \stackrel{\phi_2,\ga_2}\lra\   \cdots\ \lra\ \mu_{r-1}\ \stackrel{\phi_{r},\ga_{r}}\lra\ \mu_{r}\  \stackrel{\cfa}\lra\  \rhi=\nu,
$$
such that the class $\Phi_{\mu_r,\nu}$ has degree $\deg(\mu_r)$ and  $\phi_r\not\in\Phi_{\mu_r,\nu}$.\e

(c) \ It is the stable limit of a complete infinite MacLane-Vaqui\'e chain.
$$
\mu_0\ \stackrel{\phi_1,\ga_1}\lra\  \mu_1\ \stackrel{\phi_2,\ga_2}\lra\ \cdots\ \lra\ \mu_{n-1} \ \stackrel{\phi_{n},\ga_{n}}\lra\ \mu_{n} \  \lra\ \cdots
$$
We say that $\mu$ has \emph{finite depth} $r$, \emph{quasi-finite depth} $r$, or \emph{infinite depth}, respectively.
\end{theorem}

\begin{proof}
If $\nu$ is a valuation of depth zero, then it obviously falls in case (a).

Otherwise, let $\ga_0=\nu(x)<\infty$. The depth zero valuation $\mu_0=\omega_{0,\ga_0}$ forms a complete finite MLV-chain of length zero, whose last valuation satisfies $\mu_0<\nu$.

In order to prove the theorem, it suffices to prove the following Claim.\e

\noindent{\bf Claim. }For any complete finite MLV-chain of length $r\ge0$, whose last valuation $\mu_r$ satisfies $\mu_r<\nu$, at least one of the two following conditions hold:\e

(1) \ The valuation $\nu$ is the stable limit of some continuous family of augmentations of $\mu_r$.\e

(2) \ We may apply to $\mu_r$ one, or two, augmentation steps and build up a valuation $\mu_r<\mu\le\nu$ such that either $\mu=\nu$, or $\deg(\mu_r)<\deg(\mu)$.\e

Indeed, suppose that $\nu=\rhi$ is the stable limit of some continuous family $\aa=\cfa$ of augmentations of $\mu_r$. For any choice of $i\in A$, we may split the step $\mu_r\to\nu$ into two pieces
$$
\mu_r\ \stackrel{\chi_i,\be_i}\lra\ \rho_i\ \stackrel{\left(\rho_j\right)_{j>i}}\lra\ \rho_{\aa_{>i}}=\nu,
$$
where $\mu_r\to\rho_i$ is an ordinary augmentation, and $\nu$ is the stable limit of the continuous family $\aa_{>i}$ of augmentations of $\rho_i$.
Arguing as in the proof of Case 2 in Lemma \ref{singlestep}, we see that $\nu$ falls in case (b) of the theorem.

On the other hand, if (2) holds, then the valuation $\mu$ is the last valuation of a complete finite MLV-chain by Lemma \ref{singlestep}.  

Therefore, the iteration of this procedure shows that either $\nu$ falls in cases (a) or (b) of the theorem, or 
there exists a complete infinite MLV-chain  $\left(\mu_n\right)_{n\ge0}$ satisfying $\mu_n<\nu$ for all $n$.
The stable limit of this MLV-chain is $\nu$ by Proposition \ref{stablelimit}. Thus, $\nu$ falls in case (c).\e

Let us prove the Claim. Let $m=\deg\left(\Phi_{\mu_r,\nu}\right)$. 
If $m>\deg(\mu_r)$, then for any $\phi\in\Phi_{\mu_r,\nu}$ the augmented valuation $\mu=[\mu_r;\phi,\nu(\phi)]$ satisfies the Claim.

Suppose that $m=\deg(\mu_r)$. 
We distinguish two cases according to possible upper bounds of the totally ordered set $$A=\nu\left(\Phi_{\mu_r,\nu}\right)\subset\gn\infty.$$

\noindent{\bf The set $A$ contains a maximal element.}

Let $\ga=\mx(A)=\nu(f)$, for some $f\in \Phi_{\mu_r,\nu}$.
By Proposition \ref{gap}, $f\in\kp(\mu_r)$ and the augmented valuation $\eta=[\mu_r;f,\ga]$ satisfies $\eta\le \nu$. If $\eta=\nu$, then $\mu=\eta$ satisfies the Claim.

Suppose $\eta<\nu$, and let $\Phi_{\eta,\nu}=[\chi]_{\eta}$. By Proposition \ref{extension}, $\deg(\eta)=\deg(f)=m$. Since $\chi$ is a key polynomial for $\eta$,  we have $\deg(\chi)\ge m$. 

If $\deg(\chi)=m$, then $\eta(f)=\eta(\chi)$ by Theorem \ref{univbound}. This  contradicts the maximality of $\ga$:
$$\ga=\eta(f)=\eta(\chi)<\nu(\chi),\qquad \nu(\chi)\in A.$$

Thus, $\deg(\chi)>m$ and the augmentation $\mu=[\eta; \chi,\nu(\chi)]$ satisfies the Claim. \e


\noindent{\bf The set $A$ does not contain a maximal element.}\e



By Proposition \ref{nuIscont}, there is a continuous family $\aa=\cfa$ of augmentations of $\mu_r$ of stable degree $m$, parameterized by the set $A$, such that $\rho_i<\nu$ for all $i\in A$, and all polynomials of degree $m$ are stable. 

If $\aa$ has a stable limit $\rhi$, then $\rhi=\nu$ by Proposition \ref{stablelimit}. Thus, the Claim is satisfied.

If $\aa$ has no stable limit, then it is essential and we may consider the limit augmentation
$\mu=[\aa;\phi,\nu(\phi)]$, for some $\phi\in\kpi(\aa)$.
Clearly, $\mu\le\nu$ by the comparison of both valuations on $\phi$-expansions.

If $\mu<\nu$, then Proposition \ref{extensionlim} shows that $\deg(\mu)=\deg(\phi)>m=\deg(\mu_r)$. Thus, the valuation $\mu$ satisfies the Claim.  
\end{proof}

\begin{corollary}\label{degLim}
With the above notation, let $m_n=\deg(\mu_n)$ for all $n\ge0$. Then,
$$
\dg\left((-\infty,\nu)_{\Lambda}\right)=\begin{cases}
\{1=m_0,\,m_1,\dots,\,m_r\},& \mbox{ if $\nu$ has finite or quasi-finite depth $r$},\\
\{m_n\mid n\ge0\},& \mbox{ if $\nu$ has infinite depth }.
 \end{cases}$$
\end{corollary}

\begin{proof}
Since $(-\infty,\nu)_{\La}$ is totally ordered, we have
$$
(-\infty,\nu)_{\La}=(-\infty,\mu_0)_{\La}\;\cup\;[\mu_0,\mu_{1})_{\La}\;\cup\;\cdots\;\cup\;[\mu_{n-1},\mu_n)_{\La}\;\cup\;\cdots
$$

In equation (\ref{mu0eq}), we find a description of $(-\infty,\mu_0)_\La$. Also, Lemmas \ref{intervalOrd} and \ref{intervalLim} describe all subintervals $[\mu_n,\mu_{n+1})_{\La}$, and the  subinterval $[\mu_r,\nu)_\La$ if $\nu$ has quasi-finite depth. The result follows easily from these descriptions.
\end{proof}\e

Therefore, the length $r$ of a MLV-chain of $\nu$, and the degrees $m_n=\deg(\mu_n)$ of the nodes of the chain, are intrinsic data of $\nu$.


In particular, the  three situations of Theorem \ref{main} are mutually exclusive.

\begin{lemma}\label{mainCor}
Let $\nu$ be a valuation on $\kx$, with support $\p\in\op{Spec}(\kx)$.

It falls in case (a) of Theorem \ref{main} if and only if $\kpn\ne\emptyset$, or $\p\ne0$.

It falls in case (b) (or case (c)) of Theorem \ref{main} if and only if $\kpn=\emptyset$, $\p=0$ and $\nu$ has finite (or infinite) degree, respectively.

\end{lemma}

Indeed, suppose that $\nu$ admits a complete finite MLV-chain of length $r$. If $\ga_r=\infty$, then $\p=\phi_r\kx$. If $\ga_r<\infty$, then $\kpn\ne\emptyset$ because $\phi_r$ is a key polynomial for $\nu$.

If $\nu$ is the stable limit of a totally ordered family of valuations, it has $\p=0$ and  $\kpm=\emptyset$ by Proposition \ref{stablelimit}.

If $\nu$ falls in case (c), then clearly $\deg(\nu)=\infty$. Finally, if $\nu$ falls in case (b), the computation of $(\mu_,\nu)_{\gn}$ in Lemma \ref{intervalLim} shows that $\deg(\nu)=\deg(\mu_r)$.\e


We end this section with another immediate consequence of Theorem \ref{main}.

Let $\g\hookrightarrow \La$ be an embedding of ordered groups and let  $\La^{\com}\subset\La$ be the subgroup of all elements whose class in $\La/\g$ is a torsion element.

We say that $\La$ is a \emph{small extension} of $\g$ if $\La/\La^{\com}$ is a cyclic group.

\begin{corollary}\label{AllSmall}
Let $\nu$ be a valuation on $\kx$ extending the valuation $v$ on $K$. Then, $\gn$ is a small extension of $\g$.
\end{corollary}

\begin{proof}
If $\nu$ falls in cases (b) or (c) of Theorem \ref{main}, the extension $\g\subset\gm$ is commensurable by Proposition \ref{stablelimit} and Theorem \ref{empty}. Thus, it is a small extension.
 
 In case (a), $\gn$ is generated by $\g$ and a finite family $\ga_0,\dots,\ga_r\in\gn$, from which $\ga_0,\dots,\ga_{r-1}$ are commensurable over $\g$ (cf. section \ref{subsecDefMLV}). 
 
Thus,  $\gn$ is obtained as a commensurable extension $\gen{\g,\ga_0,\dots,\ga_{r-1}}_\Z$ of $\g$, followed by an extension with cyclic quotient.
\end{proof}

\subsection{Unicity of MacLane-Vaqui\'e chains}\label{subsecUnicity}
Let $\nu$ be a valuation on $\kx$. For any positive integer $m$, consider the set:
\begin{equation}\label{Vm}
V_m=V_m(\nu)=\left\{\nu(f)\mid f\in\kx \mbox{ monic},\ \deg(f)=m\right\}\subset\gn\infty.
\end{equation}




Consider two complete infinite MacLane-Vaqui\'e chains 
$$
\mu_0\ \stackrel{\phi_1,\ga_1}\lra\  \mu_1\ \stackrel{\phi_2,\ga_2}\lra\ \cdots
\ \lra\ \mu_{n-1} 
\ \stackrel{\phi_{n},\ga_{n}}\lra\ \mu_{n} 
\ \lra\ \cdots 
$$
$$
\mu^*_0\ \stackrel{\phi^*_1,\ga^*_1}\lra\  \mu^*_1\ \stackrel{\phi^*_2,\ga^*_2}\lra\ \cdots
\ \lra\ \mu^*_{n-1} 
\ \stackrel{\phi^*_{n},\ga^*_{n}}\lra\ \mu^*_{n} 
\ \lra\ \cdots 
$$
having the same stable limit  $\nu$. 

By Corollary \ref{degLim}, $m_n=\deg(\mu_n)=\deg(\mu^*_n)$, for all $n\ge0$. 

For any limit augmentation $\mu_n\to\mu_{n+1}$ denote by $\mu_{n,\infty}=\rhi$ the stability function of the underlying essential continuous family $\aa$ of augmentations of $\mu_n$.  
Consider an analogous notation $\mu^*_{n,\infty}$ for the second MLV-chain.

\begin{theorem}\label{unicity}
For all $n\ge0$, $\ \g_{\mu_n}=\g_{\mu^*_n}$ and the following conditions hold.\e

(1) \ If the set $V_{m_n}$ defined in (\ref{Vm}) contains a maximal element, then 
$$\mu_n=\mu^*_n,\qquad  \ga_n=\ga^*_n=\max\left(V_{m_n}\right),$$
and the augmentation steps $\mu_n\to\mu_{n+1}$, $\mu_n\to\mu^*_{n+1}$ are both ordinary.\e

(2) \ If  $V_{m_n}$ does not contain a maximal element, then  $\mu_n\to\mu_{n+1}$, $\mu^*_n\to\mu^*_{n+1}$ are  limit augmentation steps whose essential continuous families of augmentations are equivalent. 
In particular, $\mu_{n,\infty}=\mu^*_{n,\infty}$.
\end{theorem}

\begin{proof}
Suppose that $V_{m_n}$ contains a maximal element. 

Take $f\in\kx$ monic of degree $m_n$ such that $\nu(f)=\max(V_{m_n})$. 

Suppose that  $\mu_n\to\mu_{n+1}$ is a limit augmentation, and let $\aa=\cfa$, with $\rho_i=[\mu_n;\, \chi_i,\be_i]$, be the corresponding essential continuous family of augmentations of $\mu_n$. 

By the  MLV condition,  the stable degree of $\aa$ is $m_n$; thus, all polynomials in $\kx$ of degree $m_n$ are stable. 
By Proposition \ref{extensionlim}, there exists $i\in A$ such that $\rho_i(f)=\nu(f)$.

On the other hand, since $\chi_i$ is a key polynomial of degree $m_n$ for $\rho_i$, Theorem \ref{univbound} shows that $\rho_i(f)\le \rho_i(\chi_i)$, and this contradicts the maximality of $\nu(f)$:
$$
\nu(f)=\rho_i(f)\le \rho_i(\chi_i)=\be_i<\be_j=\nu(\chi_j)\in V_{m_n}.
$$

Therefore, the augmentation step $\mu_n\to\mu_{n+1}$ must be ordinary. 
By the MLV condition, 
\begin{equation}\label{primera}
\deg(f)=m_n<m_{n+1}=\deg\left(\Phi_{\mu_n,\mu_{n+1}}\right)=\deg\left(\Phi_{\mu_n,\nu}\right).
\end{equation}
Hence, $\mu_n(f)=\nu(f)$. 
Since $\phi_n\in\kp(\mu_n)$ has degree $m_n$ too,  Theorem \ref{univbound} shows that
\begin{equation}\label{segona}
\nu(f)=\mu_n(f)\le \mu_n(\phi_n)=\ga_n=\nu(\phi_n)\in V_{m_n}.
\end{equation}
Since $\nu(f)$ is maximal, we must have $\ga_n=\nu(f)=\max(V_{m_n})$.

This argument is valid for both MLV-chains. Thus, the augmentation steps $\mu_n\to\mu_{n+1}$, $\mu^*_n\to\mu^*_{n+1}$ are both ordinary and $\ga_n=\ga^*_n=\max(V_{m_n})$.

Finally, it is obvious that $\mu_n\le\mu_n^*$ by comparing the action of both valuations on $\phi_n$-expansions. A symmetric argument shows that $\mu_n=\mu^*_n$.\e

Now, suppose that $V_{m_n}$ does not contain a maximal element.

If the augmentation $\mu_n\to\mu_{n+1}$ is ordinary, equations (\ref{primera}) and (\ref{segona}) hold for any monic polynomial $f$ of degree $m_n$. They imply $\max(V_{m_n})=\ga_n$, contradicting the hypothesis. 

Thus,  $\mu_n\to\mu_{n+1}$ is a limit augmentation. 

Since $m_n=\deg(\mu_n)=\deg(\mu^*_n)$, this argument applies to the second chain too and  $\mu^*_n\to\mu^*_{n+1}$ is a limit augmentation. Let $\aa$, $\aa^*$ be the underlying essential continuous families of augmentations of $\mu_n$ and $\mu^*_n$, respectively. 

By the MLV condition, the stable degree of $\aa$ and $\aa^*$ is $m_n$.
Since $\aa$ and $\aa^*$ admit the valuation $\nu$ as a common upper bound, Lemma \ref{critEquiv} shows that they are equivalent.

Finally, since the classes $\Phi_{\mu_n,\mu_{n+1}}$ and $\Phi_{\mu^*_n,\mu^*_{n+1}}$ are proper, Lemma \ref{mu=rinf} shows that $\g_{\mu_n}=\g_\aa=\g_{\aa^*}=\g_{\mu^*_n}$. 
\end{proof}\e

These arguments yield completely analogous unicity results for MacLane-Vaqui\'e chains of valuations of finite, or quasi-finite depth.



Let $\left(\mu_n\right)_{0\le n}$ be a MLV-chain of $\nu$ of length $r\in\N\infty$. The different behaviour of ordinary and limit augmentations, as far as unicity is concerned, may be explained too by an analysis of the fibers of the order-preserving degree map
$$
\deg\colon (-\infty,\nu)_\La \lra \N.
$$

By Corollary \ref{degLim}, a non-empty fiber is of the form $\deg^{-1}(m_n)$ for some $n\ge0$. If this fiber has a maximal element, this element must be a node $\mu_n$ of the chain, and $\mu_n\to\mu_{n+1}$ is necessarily an ordinary augmentation. If this fiber has no maximal element, there is no canonical node  of degree $m_n$ in the chain.   

The following graphic of the degree function illustrates the situation.

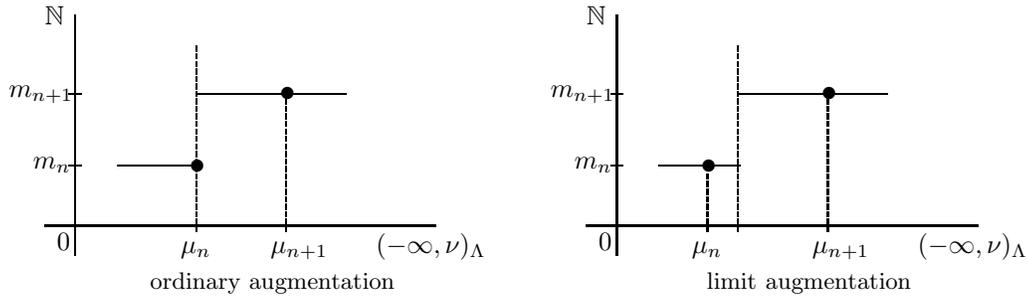
\begin{center}
\setlength{\unitlength}{4mm}
\begin{picture}(30,10)
\put(3.8,3.85){$\bullet$}\put(6.8,6.25){$\bullet$}
\put(-1,2.1){\line(1,0){13}}\put(0,1.1){\line(0,1){8}}
\put(1.4,4.1){\line(1,0){2.7}}\put(4,6.5){\line(1,0){5}}
\put(-.2,4.1){\line(1,0){0.4}}\put(-.2,6.5){\line(1,0){0.4}}
\multiput(4,1.9)(0,.25){25}{\vrule height2pt}
\multiput(7,1.9)(0,.25){18}{\vrule height2pt}
\put(10,1.2){\begin{footnotesize}$(-\infty,\nu)_\La$\end{footnotesize}}
\put(3.5,1.2){\begin{footnotesize}$\mu_n$\end{footnotesize}}
\put(6.5,1.2){\begin{footnotesize}$\mu_{n+1}$\end{footnotesize}}
\put(-1.4,4){\begin{footnotesize}$m_n$\end{footnotesize}}
\put(-.6,1.3){\begin{footnotesize}$0$\end{footnotesize}}
\put(-2.2,6.4){\begin{footnotesize}$m_{n+1}$\end{footnotesize}}
\put(-1,8.8){\begin{footnotesize}$\N$\end{footnotesize}}
\put(2.5,0){\begin{scriptsize}ordinary augmentation\end{scriptsize}}

\put(20.8,3.85){$\bullet$}\put(24.8,6.25){$\bullet$}
\put(17,2.1){\line(1,0){13}}\put(18,1.1){\line(0,1){8}}
\put(19.4,4.1){\line(1,0){2.7}}\put(22,6.5){\line(1,0){5}}
\put(17.8,4.1){\line(1,0){0.4}}\put(17.8,6.5){\line(1,0){0.4}}
\multiput(21,1.9)(0,.25){8}{\vrule height2pt}
\multiput(22,1.9)(0,.25){25}{\vrule height2pt}
\multiput(25,1.9)(0,.25){18}{\vrule height2pt}
\put(28,1.2){\begin{footnotesize}$(-\infty,\nu)_\La$\end{footnotesize}}
\put(20.5,1.2){\begin{footnotesize}$\mu_n$\end{footnotesize}}
\put(24.5,1.2){\begin{footnotesize}$\mu_{n+1}$\end{footnotesize}}
\put(16.6,4){\begin{footnotesize}$m_n$\end{footnotesize}}
\put(17.4,1.3){\begin{footnotesize}$0$\end{footnotesize}}
\put(15.8,6.4){\begin{footnotesize}$m_{n+1}$\end{footnotesize}}
\put(17,8.8){\begin{footnotesize}$\N$\end{footnotesize}}
\put(21,0){\begin{scriptsize}limit augmentation\end{scriptsize}}
\end{picture}
\end{center}\bs

\section{Structure of $\ggn$ as a $\gg_v$-algebra}\label{secGmu}
Let $\nu$ be a valuation on $\kx$. 
By Theorem \ref{main}, there is a depth zero valuation $\mu_0=\omega_{a_0,\ga_0}$ and a finite, or countably infinite, MacLane-Vaqui\'e chain
of length $r\in\N\infty$,
\begin{equation}\label{depthMLVGr}
\mu_0\ \stackrel{\phi_1,\ga_1}\lra\  \mu_1\ \stackrel{\phi_2,\ga_2}\lra\ \cdots
\ \lra\ \mu_{n-1} 
\ \stackrel{\phi_n,\ga_n}\lra\ \mu_n 
\ \lra\ \cdots
\end{equation}
such that $\nu$ falls in one of the following three cases:\e

(a) \ $\nu=\mu_r$,\vskip0.1cm

(b) \ $\nu=\rhi$ is the stable limit of a continuous family $\aa$ of augmentations of $\mu_r$, of stable degree $\deg(\mu_r)$. Moreover, $\phi_r\not\in\Phi_{\mu_r,\nu}$.\vskip0.1cm

(c) \ $\nu=\lim_{n\to\infty}\mu_n$ is the stable limit of the infinite  MacLane-Vaqui\'e chain.\e

Let $\g_{-1}=\g$ and $\phi_0=x-a_0$. For all $n\ge0$, denote $m_n=\deg(\mu_n)=\deg(\phi_n)$ and
$$
\g_n=\g_{\mu_n},\qquad e_n=\left(\g_n\colon\g_{n-1}\right),\qquad \Delta_n=\Delta_{\mu_n},\qquad \ka_n=\ka(\mu_n).
$$

Since $\mu_n$ has finite depth, the following observation follows from Lemma \ref{mainCor}.

\begin{lemma}\label{empty2}\mbox{\null}
\begin{enumerate}
\item $\kp(\mu_n)=\emptyset$ if and only if $n=r$ and $\ga_r=\infty$.
\item $\mu_n/v$ is incommensurable if and only if $n=r$ and $\ga_r\not\in\gq\infty$.
\end{enumerate}
\end{lemma}

If $\op{KP}(\mu_n)\ne\emptyset$, we consider the following element in the graded algebra $\gg_{\mu_n}$:
$$q_n=\op{in}_{\mu_n}\phi_n\qquad\qquad\ \ \ \mbox{homogeneous prime element  of degree }\ga_n.$$

If $\mu_n$ is residually transcendental (that is, $\op{KP}(\mu_n)\ne\emptyset$ and $\mu_n/v$ commensurable), then we consider the following elements in $\gg_{\mu_n}$:
$$
\as{1.3}
\begin{array}{ll}
u_n & \qquad\mbox{homogeneous unit of degree }e_n\ga_n\in\g_{n-1},\\
\xi_n=q_n^{e_n}u_n^{-1}\in\Delta_n& \qquad\mbox{Hauptmodul of $\Delta_n$ over $\ka_n$}. 
\end{array}
$$

Let $\rho$ be any valuation such that $\mu_n<\rho\le\nu$.
By Corollary \ref{KerImGr}, the canonical homomorphism $\gg_{\mu_n}\to\gg_\rho$ maps the three elements $q_n$, $u_n$, $\xi_n$ to homogeneous units in $\gg_\rho$, which we denote by $x_n,\,u_n,\,z_n$, respectively. 

In this notation, we omit the reference to the valuation $\rho$. Actually, this will be a general convention on the notation of units of the graded algebras. \e

\noindent{\bf Convention. }Let $\eta<\rho$ be two valuations on $\kx$. Given a unit $u\in\gg_{\eta}^*$, we denote by the same symbol $u\in\gg_\rho^*$ the  image of $u$ under the canonical homomorphism $\gg_{\eta}\to\gg_{\rho}$.\e

In particular, if $\mu_n<\nu$, we get homogeneous units in $\ggn$:\vskip0.15cm

\begin{itemize}
\item $q_n\longmapsto x_n=\inm(\phi_n)$ \quad\qquad\qquad\qquad of degree $\ga_n$.
\item $u_n\longmapsto u_n$ \quad\ \qquad\qquad\qquad\qquad\qquad \,of degree $e_n\ga_n\in\g_{n-1}$.
\item $\xi_n\longmapsto z_n=x_n^{e_n}u_n^{-1}\in\ka(\nu)^*$ \qquad\!\qquad of degree zero. 
\end{itemize}\vskip0.15cm

Our aim, in this section, is to describe the structure of $\ggn$ as a $\gg_v$-algebra, in terms of the discrete data supported by any MacLane-Vaqui\'e chain of $\nu$. 

\subsection{Computation of $\ka(\nu)$ and the residue class field $\kn$}

Let $\ka=\ka(\nu)$ be the algebraic closure of $k$ in $\Delta=\Delta_\nu$. 

The canonical homomorphisms $\gg_{\mu_{n}}\to\,\gg_{\mu_{n+1}}\to\,\ggn$ induce a tower of field extensions
\begin{equation}\label{kappaTower} 
k=\ka_0\;\subset\;\ka_1\;\subset\; \cdots \;\subset\;  \ka_n\;\subset\; \cdots \;\subset\;\ka.
\end{equation}
Our first aim is to show that each field extension $\ka_{n+1}/\ka_n$ is finite.

\begin{lemma}\label{MinPolZi}
Suppose that $\mu_n<\nu$ and let $\Phi_{\mu_n,\nu}=[\phi]_{\mu_n}$. Then, the unit $z_n\in\ggn^*$ is algebraic over $\ka_n$ and its minimal polynomial over $\ka_n$ is the residual polynomial $R_{\mu_n,\phi_n,u_n}(\phi)$. 

In particular, $[\ka_n(z_n)\colon\ka_n]=\deg\left(\Phi_{\mu_n,\nu}\right)/e_nm_n$.
\end{lemma}

\begin{proof}
If $\mu_n\to\mu_{n+1}$ is an ordinary augmentation with $\mu_{n+1}\le\nu$, then we may take $\phi=\phi_{n+1}$. By the MLV condition, $\deg(\phi_n)<\deg(\phi)$, so that $\phi_n\nmid_{\mu_n}\phi$.

Otherwise, there is a continuous family $\aa$ of augmentations of $\mu_n$ such that, either $\aa$ is  essential and $\mu_{n+1}=[\aa;\,\phi_{n+1},\ga_{n+1}]$ is a limit augmentation with $\mu_{n+1}\le\nu$, or $n=r$ and $\nu=\rhi$ is the stable limit of $\aa$ as indicated in case (b) of Theorem \ref{main}. 

In both cases, the class $\Phi_{\mu_n,\nu}=[\phi]_{\mu_n}$ has degree $m_n$ and $\phi_n$ does not belong to this class. Thus, we have again $\phi_n\nmid_{\mu_n}\phi$, by Corollary \ref{mid=sim}.

In all three cases, $\phi_n$ is a key polynomial for $\mu_n$ of minimal degree, and $\phi$ is a key polynomial for $\mu_n$ such that $\phi_n\nmid_{\mu_n}\phi$. 

Consider  the residual polynomial operator attached to $\mu_n$ in section \ref{subsecKP}:
$$R=R_{\mu_n,\phi_n,u_n}\colon \kx\to\ka_n[y].$$ 
By \cite[Thm. 5.3]{KP}, $\op{in}_{\mu_n} \phi$ admits the following decomposition in the graded algebra
$$
\op{in}_{\mu_n} \phi=\ep\,R(\phi)(\xi_n) \quad \mbox{ for some unit }\ep\in\gg^*_{\mu_n}.
$$
 By Corollary \ref{KerImGr}, the homomorphism $\gg_{\mu_n}\to\ggn$ vanishes on $\op{in}_{\mu_n} \phi$. Thus, it vanishes on  $R(\phi)(\xi_n)$ too; that is,  $R(\phi)(z_n)=0$.
This ends the proof, because $R(\phi)\in\ka_n[y]$ is a monic irreducible polynomial of degree $\deg(\phi)/e_nm_n$,  by Theorem \ref{charKP}.
\end{proof}

\begin{lemma}\label{kappaOrd}
For any augmentation step $\mu_n\to\,\mu_{n+1}\le\nu$, we have $\ka_{n+1}=\ka_n(z_n)$.

Moreover, if $\mu_n\to\,\mu_{n+1}$ is a limit augmentation, then $\ka_{n+1}=\ka_n(z_n)=\ka_n$.
\end{lemma}

\begin{proof}
Let $\ep\in\ka_{n+1}^*=\Delta_{n+1}^*$. If $\kp(\mu_{n+1})\ne\emptyset$, then  Lemma \ref{units} shows that a unit of degree zero in the graded algebra may be written as
\begin{equation}\label{ep=}
\ep=\op{in}_{\mu_{n+1}}a,\quad a\in\kx,\quad \deg(a)<m_{n+1},\quad \mu_{n+1}(a)=0.
\end{equation}

If $\kp(\mu_{n+1})=\emptyset$, then $\mu_{n+1}=\nu$ and $\ga_{n+1}=\infty$, by Lemma \ref{empty2}. Thus, $\nu$ may be identified with a valuation on the field $\kx/\phi_{n+1}\kx$, of degree $m_{n+1}$ over $K$. By Theorem \ref{empty}, $k=\Delta=\kn$, so that the equality (\ref{ep=}) holds trivially too.  

Suppose that the augmentation $\mu_n\to\,\mu_{n+1}$ is ordinary.
Since $\phi_{n+1}\nmid_{\mu_n}a$, Proposition \ref{extension} shows that $\mu_n(a)=\mu_{n+1}(a)=0$. Hence, $\op{in}_{\mu_n}a$ belongs to $\Delta_n$, and $\ep=\op{in}_{\mu_{n+1}}a$ belongs to the image of the ring homomorphism $\Delta_n\to\Delta_{n+1}$. By Theorem \ref{comm}, $\Delta_n=\ka_n[\xi_n]$, and we deduce that
$$
\ep\in\im\left(\Delta_n\to\Delta_{n+1}\right)=\ka_n[z_n]=\ka_n(z_n).
$$

Now, suppose that  $\mu_n\to\,\mu_{n+1}$ is a limit augmentation with respect to  $$\aa=\cfa, \qquad \rho_i=[\mu_n;\,\chi_i,\be_i],$$ an essential continuous family of augmentations of $\mu_n$, of stable degree $m_n$. 

The condition $\deg(a)<m_{n+1}=\deg(\phi_{n+1})$ implies that $a$ is $\aa$-stable. Hence, there exists $i\in A$ such that $\rho_i(a)=\mu_{n+1}(a)=0$.
This means that $\ep$ is the image of the unit $\op{in}_{\rho_{i}}(a)\in\ka(\rho_{i})$.

On the other hand, if we apply to the augmentation $\mu_n\to\rho_i$ the arguments used in the ordinary case,  we obtain   
$\ka(\rho_{i})=\ka_n(z_n)$ as well, where $z_n\in\ka(\rho_i)$ is the image of $\xi_n$ under the homomorphism $\gg_{\mu_n}\to\gg_{\rho_i}$.

The relative ramification index of $\rho_i$ is equal to one (section \ref{subsecCFA}).
Since $$\Phi_{\mu_n,\mu_{n+1}}=\Phi_{\mu_n,\rho_i}=[\chi_i]_{\mu_n},$$ Lemma \ref{MinPolZi} shows that $z_n\in\ka(\rho_i)$ has degree $m_n/m_n=1$ over $\ka_n$.

Therefore,  $\ka(\rho_{i})=\ka_n(z_n)=\ka_n$, and $\ep$ is the image of some element in $\ka_n^*$.   
\end{proof}

\begin{theorem}\label{kappa}
If $\nu=\mu_r$ has finite depth, then
$$
\ka=k(z_0,\dots,z_{r-1}),\qquad \kn=
\begin{cases}
\ka(\xi_r),&\mbox{ if }\ga_r\in\gq,\\
\ka,&\mbox{ if }\ga_r\not\in\gq.
\end{cases}
$$

Suppose that $\nu=\rhi$ is the stable limit of a continuous family $\aa$ of augmentations of $\mu_r$, of stable degree $m_n$ and such that $\phi_r\not\in\Phi_{\mu_r,\nu}$. Then,
$$
\ka=k(z_0,\dots,z_{r-1})=\kn.
$$

If $\nu=\lim_{n\to\infty}\mu_n$ is the stable limit of an infinite  MacLane-Vaqui\'e chain, then
$$
\ka=k(z_0,\dots,z_{i},\dots)=\kn.
$$
\end{theorem}

\begin{proof}
If $\nu=\mu_r$ has finite depth, then $\ka=\ka_r=k(z_0,\dots,z_{r-1})$ by Lemma \ref{kappaOrd}.

If $\nu$ has quasi-finite or infinite depth, then $\kpm=\emptyset$ and every non-zero homogeneous element in $\ggn$ is a unit, by Proposition \ref{stablelimit}. Hence, 
any $\ep\in\ka^*=\Delta^*$ may be written as
$$
\ep=\inn f,\qquad f\in\kx,\qquad  \nu(f)=0.
$$

Suppose that $\nu$ is the stable limit of a continuous family $\aa=\cfa$ of augmentations of $\mu_r$, of stable degree $m_r$ such that $\phi_r\not\in\Phi_{\mu_r,\nu}$.

Since $f$ is stable, there exists an index $i\in A$ such that $\rho_i(f)=\nu(f)=0$. 
Thus, $\ep$ is the image of the unit $\op{in}_{\rho_i}f\in\ka(\rho_{i})$.
Now, the arguments in the proof of Lemma \ref{mainCor} show that $\ka(\rho_{i})=\ka_r$. Therefore, we have again $\ka=\ka_r=k(z_0,\dots,z_{r-1})$.

Finally, if $\nu=\lim_{n\to\infty}\mu_n$, there exists an index $n\ge0$ such that $\mu_n(f)=\nu(f)=0$. 
Thus, $\ep$ is the image of the unit $\op{in}_{\mu_n}f\in\ka_n$. Therefore, in this case we have
$$
\ka=\bigcup\nolimits_{n\ge0}\ka_n.$$

The statements about $\kn$ follow from Proposition \ref{stablelimit} and Theorems \ref{empty}, \ref{incomm}, \ref{comm}.
\end{proof}


\begin{corollary}\label{uniquekappa}
Up to natural isomorphisms, the tower (\ref{kappaTower}) of finite field extensions
is independent of the MacLane-Vaqui\'e chain underlying $\nu$.
\end{corollary}


\subsection{Structure of $\ggn$ as a $\gg_v$-algebra}
We keep the notation $\ka=\ka(\nu)$ and $\Delta=\Delta_\nu$. 

The embedding of graded $k$-algebras $\gg_v\hookrightarrow\ggn$ induces an embedding of graded $\ka$-algebras:
$$
\left(\gg_v\otimes_k\ka\right)\hooklongrightarrow\ggn.
$$

In this section, we find an explicit description of $\ggn$ as a $\left(\gg_v\otimes_k\ka\right)$-algebra. Together with
Theorem \ref{kappa}, which computes $\ka$ in terms of discrete data of the MacLane-Vaqui\'e chain of $\nu$, we obtain an explicit description of the structure of $\ggn$ as a $\gg_v$-algebra.

\begin{lemma}\label{intrinsice}
Let $\mu_n<\nu$ be a node of the MacLane-Vaqui\'e chain of $\nu$.
For any $\alpha\in\g_n$, there exists a homogeneous unit $u\in\gg_{\mu_n}^*$ of degree $\al$ if and only if $\alpha\in\g_{n-1}$.
In this case, $$\left(\pset_\al(\mu_n)/\pset^+_\al(\mu_n)\right)\cap \gg_{\mu_n}^*=\ka^*_nu.$$
\end{lemma}

\begin{proof}
As we saw in section \ref{subsecDefMLV}, $\g_{n-1}=\g_{n,m_n}$. Thus, for all
 $\alpha\in\g_{n-1}$ there is $a\in\kx$  such that $\deg(a)<m_n$ and $\mu_n(a)=\alpha$.
By Lemma \ref{units}, $\op{in}_{\mu_n}a$ is a unit of degree $\al$.

On the other hand, by (\ref{chainGamma}), any $\alpha\in\g_n\setminus\g_{n-1}$ can be written as 
$$\alpha=\ell\ga_n+\beta, \qquad \ell\in\Z,\quad\ell\ne0, \qquad\beta\in\g_{n-1}.
$$
If $\ell<0$,  Lemma \ref{minimal0} shows that $\pset_\al=0$. Suppose that $\ell>0$.

By the previous argument, there exists a unit $u\in\gg_{\mu_n}^*$ of degree $\beta$. Then, $u\,q_n^\ell$ has degree $\alpha$. Since $q_n$ is a prime element in $\gg_{\mu_n}$, there is no unit in $$\pset_\al/\pset^+_\al=\left(u\,q_n^\ell\right)\Delta_n.$$

Finally, if $u\in\gg_{\mu_n}^*$ is a homogeneous unit of degree $\al$, we have
$$
\left(\pset_\al/\pset^+_\al\right)\cap \gg_{\mu_n}^*=\left(\Delta_nu\right)\cap \gg_{\mu_n}^*=\Delta^*_nu=\ka^*_nu.
$$
This ends the proof of the lemma.
\end{proof}

\begin{theorem}\label{ggm}
(1) \ If $\nu=\mu_r$ has finite depth, then
$$
\ggn=\begin{cases}
\left(\gg_v\otimes_k\ka\right)[x_0,\dots,x_{r-1}],&\mbox{ if }\ga_r=\infty,\\
\left(\gg_v\otimes_k\ka\right)[x_0,\dots,x_{r-1}][q_r],&\mbox{ if }\ga_r<\infty.
\end{cases}
$$
If $\ga_r<\infty$, then $q_r$ is a homogeneous prime element of degree $\ga_r$ which is transcendental over $\left(\gg_v\otimes_k\ka\right)[x_0,\dots,x_{r-1}]$. 

(2) \ Suppose that $\nu$ is the stable limit of a continuous family of augmentations of $\mu_r$, of stable degree $m_r$, such that $\phi_r\not\in\Phi_{\mu_r,\nu}$.
 Then,
$$
\ggn=\left(\gg_v\otimes_k\ka\right)[x_0,\dots,x_r].
$$

(3) \ If $\nu=\lim_{n\to\infty}\mu_n$ is the stable limit of an infinite  MacLane-Vaqui\'e chain, then
$$
\ggn=\left(\gg_v\otimes_k\ka\right)[x_0,\dots,x_n,\dots]. 
$$

In all cases, $x_n$ is a homogeneous unit of degree $\ga_n$, algebraic over $$\left(\gg_v\otimes_k\ka\right)[x_0,\dots,x_{n-1}],$$ with minimal equation 
\begin{equation}\label{xizi}
x_n^{e_n}=u_nz_n\in \left(\gg_v\otimes_k\ka\right)[x_0,\dots,x_{n-1}],\qquad n\ge0.
\end{equation}
\end{theorem}

\begin{proof}
For $n>0$, since $\g_{n-1}=\gen{\g,\ga_0,\dots,\ga_{n-1}}$ and $x_0,\dots,x_{n-1}$ are units of degree $\ga_0,\dots,\ga_{n-1}$ respectively, the subalgebra $$\left(\gg_v\otimes_k\ka\right)[x_0,\dots,x_{n-1}]\subset\ggn$$ contains a homogeneous unit $u_\al$ of degree $\al$, for all $\al\in\g_{n-1}$.
By Lemma \ref{intrinsice}, this algebra contains all homogeneous units of degree $\al$, for all $\al\in\g_{n-1}$. 
Thus, it contains $u_n$ and $z_n$.

If $\nu=\mu_r$ has finite depth and  $\ga_r=\infty$, we have $\gn=\g_{r-1}$ and $\Delta=\ka$. For all $\al\in\gn$,
$$
\pset_\al(\nu)/\pset^+_\al(\nu)=\Delta\,u_\al=\ka\,u_\al\subset \left(\gg_v\otimes_k\ka\right)[x_0,\dots,x_{r-1}].
$$
Hence, $\left(\gg_v\otimes_k\ka\right)[x_0,\dots,x_{r-1}]=\ggn$.

If $\nu=\mu_r$ has finite depth and $\ga_r<\infty$,  consider the subalgebra $$\gg=\left(\gg_v\otimes_k\ka\right)[x_0,\dots,x_{r-1}][q_r]\subset\ggn.$$

If $\nu/v$ is incommensurable, then $\Delta=\ka$ by Theorem \ref{incomm}. 
If $\nu/v$ is commensurable, the algebra $\gg$ contains $u_r$  by the arguments above. Hence, it contains $\xi_r=q_r^{e_r}u_r^{-1}$. Therefore, in all cases 
$\Delta\subset\gg\subset \ggn$.

Now, let $\al\in\g_r=\gn$ such that $\pset_\al/\pset^+_\al\ne0$. If  $\inn f$ is a non-zero homogeneous element of degree $\al$, Lemma \ref{minimal0} shows that
$$\al=\nu(f)=\ell \ga_r+\be, \qquad\ell\in\Z_{\ge0},\qquad\be\in\g_{r-1}.$$ If $u$ is an homogeneous unit in $\left(\gg_v\otimes_k\ka\right)[x_0,\dots,x_{r-1}]$ of degree $\be$, the element $\zeta=q_r^\ell u$ belongs to $\gg$ and has degree $\al$.

Hence, $\pset_\al/\pset^+_\al=\Delta \zeta$ is contained in $\gg$. This proves that $\gg=\ggn$.

This proves (1). The other two cases follow from similar arguments. Actually, the proof is easier in both cases, because  $\Delta=\ka$ and all non-zero homogeneous elements are units.
\end{proof}

\begin{corollary}\label{algclosure}
If $\nu$ has finite depth and trivial support, the subalgebra $$\gg_0=\left(\gg_v\otimes_k\ka\right)[x_0,\dots,x_{r-1}]$$ is the algebraic closure of $\ggv$ in $\ggn=\gg_0[q_r]$, which is purely transcendental over $\gg_0$.

In all other cases, the extension $\ggv\subset\ggn$ is algebraic.  
\end{corollary}


We may reinterpretate Theorem \ref{ggm} as an intrinsic construction of $\ggn$, depending only on the data $e_n,\ga_n$ for all $n\ge0$. 

For instance, suppose  that $\nu$ has finite depth $r$. We consider indeterminates $x_0,\dots,x_{r-1}$ of degrees $\ga_0,\dots,\ga_{r-1}$, submitted to the relations (\ref{xizi}).

For $n<r-1$, once the subalgebra $\left(\gg_v\otimes_k\ka\right)[x_0,\dots,x_{n-1}]$ has been constructed, we may choose an arbitrary homogeneous unit $u_n$ of degree $e_n\ga_n\in\g_{n-1}$ in this subalgebra, and this facilitates the construction of $\left(\gg_v\otimes_k\ka\right)[x_0,\dots,x_n]$ .

If $\ga_r=\infty$ we get $\ggn=\left(\gg_v\otimes_k\ka\right)[x_0,\dots,x_{r-1}]$.

If $\ga_r<\infty$, we consider an indeterminate $q_r$ of degree $\ga_r$ and  we take    $$\ggn=\left(\gg_v\otimes_k\ka\right)[x_0,\dots,x_{r-1}][q_r].$$  


\section{Defect of the extension of a valuation}\label{secDef}

\subsection{Numerical data attached to a valuation on $\kx$}
Let $\nu$ be a  valuation on $\kx$ with underlying MacLane-Vaqui\'e chain as in (\ref{depthMLVGr}).

Keeping with all notation introduced in section \ref{secGmu}, we may attach to each node $\mu_n<\mu_{n+1}\le\nu$ of the chain  the following numerical data.

The  \emph{relative ramification index} and \emph{residual degree}, defined as the positive integers
$$
e_n=\left(\g_n\colon\g_{n-1}\right),\qquad f_n=[\ka_{n+1}\colon\ka_n]=\deg\left(\Phi_{\mu_n,\nu}\right)/e_nm_n,
$$
respectively. Also, the  \emph{relative gap} is defined as the rational number
$$
d_n=\dfrac{m_{n+1}}{\deg\left(\Phi_{\mu_n,\nu}\right)}=
\begin{cases}
1,&\mbox{ if }\mu_{n}\to\mu_{n+1}\mbox{ is an ordinary augmentation},\\
m_{n+1}/m_n,&\mbox{ if }\mu_{n}\to\mu_{n+1}\mbox{ is a limit augmentation}.
\end{cases}
$$


By Theorem \ref{unicity} and Corollary \ref{uniquekappa}, these numbers are intrinsic data of $\nu$. Clearly,
\begin{equation}\label{degdeg2}
m_{n+1}=e_nf_nd_nm_n\quad\mbox{for all}\quad n\ge0.
\end{equation}

For all $n\ge0$ we may consider the valuation 
$$
v_n:=[\mu_n;\phi_n,\infty],
$$
with support $\p_{v_n}=\phi_n\kx$. This valuation determines an extension of $v$ to the finite field extension $K_{\phi_n}/K$, where $K_{\phi_n}=\kx/\left(\phi_n\kx\right)$. We abuse of language and use the same symbol $v_n$ for both valuations. That is,
$$
v_n\colon \kx\twoheadrightarrow K_{\phi_n}\stackrel{v_n}\lra \g_{\mu_n}\infty, 
$$

It is well known how to compute the value group and residue class field of this valuation in terms of data of the pair $\mu_n,\phi_n$ \cite[Props. 2.12+3.6]{KP}: 
\begin{equation}\label{vi}
\g_{v_n}=\g_{\mu_n,m_n}=\g_{n-1},\qquad  k_{v_n}\simeq \ka_n. 
\end{equation}
Let us define $d(\mu_n/v):=m_n/e(v_n/v)f(v_n/v)$, where $e(v_n/v)$, $f(v_n/v)$ are the ramificaction index and residual degree of the extension $v_n/v$, respectively. Then, we have:
\begin{equation}\label{CompDefect}
e(v_n/v)=e_0\cdots e_{n-1},\qquad  f(v_n/v)=f_0\cdots f_{n-1},\qquad 
d(\mu_n/v)=d_0\cdots d_{n-1}. 
\end{equation}

In fact, the formulas for $e(v_n/v)$ and $f(v_n/v)$ follow immediately from  (\ref{vi}) and the definition of the numbers $e_n$, $f_n$. The formula for $d(\mu_n/v)$ follows from  (\ref{degdeg2}).

\subsection{Defect of a valuation on a simple finite extension of fields}

Let $L/K$ be a finite simple field extension. That is, $L=K(\t)$, where $\t$ is the root of some monic irreducible $F\in \kx$ in some fixed algebraic closure of $K$.

Let $\nu$ be a valuation on $L$ extending $v$. We abuse of language and use the same symbol to denote the valuation on $\kx$ determined by $\nu$.


Since $\nu$ has non-trivial support $\p=F\kx$, Lemma \ref{mainCor} shows that $\nu$ has finite depth and admits a finite MacLane-Vaqui\'e chain
$$
\mu_0\ \stackrel{\phi_1,\ga_1}\lra\  \mu_1\ \stackrel{\phi_2,\ga_2}\lra\ \cdots
\ \lra\ \mu_{r-1} 
\ \stackrel{\phi_{r},\ga_{r}}\lra\ \mu_{r}=\nu 
$$
with $\ga_r=\infty$ and $\phi_r=F$.
By equations (\ref{degdeg2}) and (\ref{CompDefect}), we have
\begin{equation}\label{Def} 
[L\colon K]=\deg(\phi_r)=e(\nu/v)f(\nu/v)d(\nu/v).
\end{equation}

If $\eta_1,\dots,\eta_g$ are the different extensions of $v$ to $L$, the well known inequality
$$
\sum_{i=1}^g e(\eta_i/v)f(\eta_i/v)\le [L\colon K]
$$
gives a hint about the information contained in the rational number $d(\nu/v)$.

In particular, if $v$ admits a unique extension $\nu$  to $L$, then the equality in (\ref{Def}) shows that $d(\nu/v)$ coincides with the \emph{defect} of the extension $\nu/v$. 

In this case, $d(\nu/v)=p^a$ is an integer, where $p=1$ if $k$ has characteristic zero and $p$ is the characteristic of $k$ if it is positive.

Hence, the computation of (\ref{CompDefect}) yields the following result of Vaqui\'e in \cite{VaqDef}.

\begin{corollary}\label{mainDef}
If $\nu$ is the only extension of $v$ to $L$, the defect of the extension $\nu/v$ is the product of the relative gaps of any MacLane-Vaqui\'e chain of $\nu$.
\end{corollary}





\begin{thebibliography}{}
\bibitem{APZ} V. Alexandru, N. Popescu, A. Zaharescu, \emph{All valuations on $K(x)$}, J. Math. Kyoto Univ.  {\bf 30}, Number 2 (1990), 281--296.

\bibitem{hos}F.J. Herrera Govantes, M.A. Olalla Acosta, M. Spivakovsky, \emph{Valuations in algebraic field extensions}, J. Algebra {\bf 312} (2007), no. 2, 1033--1074.

\bibitem{inmos}F.J. Herrera Govantes, W. Mahboub, M.A. Olalla Acosta, M. Spivakovsky, \emph{Key polynomials for simple extensions of valued fields}, preprint, arXiv:1406.0657v4 [math.AG], 2018.


\bibitem{Kuhl}F.-V. Kuhlmann, \emph{Value groups, residue fields, and bad places of rational function fields}, Trans. Amer. Math. Soc. {\bf 356} (2004), no. 11, 4559--4660. 

\bibitem{mcla} S. MacLane, \emph{A construction for absolute values in polynomial rings}, Trans. Amer. Math. Soc. {\bf40} (1936), pp. 363--395.

\bibitem{mclb} S. MacLane, \emph{A construction for prime ideals as absolute values of an algebraic field}, Duke Math. J. {\bf2} (1936), pp. 492--510.

\bibitem{defless} N. Moraes de Oliveira, E. Nart, \emph{Defectless polynomials over henselian fields and inductive valuations}, J. Algebra {\bf541} (2020), 270--307.



\bibitem{KP} E. Nart, \emph{Key polynomials over valued fields}, Publ. Mat. {\bf 64} (2020), 3--42.  


\bibitem{VaqJap}M. Vaqui\'e, \emph{Famille admisse associ\'ee \`a une  valuation de $\kx$}, Singularit\'es Franco-Japonaises, S\'eminaires et Congr\'es 10, SMF, Paris (2005), Actes du colloque franco-japonais, juillet 2002, \'edit\'e par Jean-Paul Brasselet et Tatsuo Suwa, 391--428.


\bibitem{Vaq}
M. Vaqui\'e, \emph{Extension d'une valuation}, Trans. Amer. Math. Soc.  {\bf 359} (2007), no. 7, 3439--3481.

\bibitem{Vaq2}
M. Vaqui\'e, \emph{Alg\`ebre gradu\'ee associ\'ee \`a une valuation de $\kx$}, Adv. Stud. Pure Math. {\bf 46} (2007), 259--271.
\bibitem{VaqDef}M. Vaqui\'e, \emph{Famille admissible de valuations et d\'efaut d'une extension}, J. Algebra {\bf 311} (2007), no. 2, 859--876.


\end{thebibliography}
\end{document}